\documentclass[twoside,12pt]{amsart}
\usepackage{fullpage}
\usepackage[safe]{tipa}
\usepackage[T1]{fontenc}

\numberwithin{equation}{section}

\newtheorem{Lemma}[equation]{Lemma}
\newtheorem{Proposition}[equation]{Proposition}
\newtheorem{Theorem}[equation]{Theorem}
\newtheorem{Corollary}[equation]{Corollary}

\theoremstyle{definition}
\newtheorem{Remark}[equation]{Remark}

\def\iso{\cong}

\def\lan{\langle}
\def\ran{\rangle}

\def\bar{\overline}

\def\ad{{\operatorname{ad}}}
\def\Ann{\operatorname{Ann}}
\def\BV{\operatorname{BV}}
\def\RS{\operatorname{RS}}
\def\DT{\operatorname{DT}}

\def\C{{\mathbb C}}
\def\Z{{\mathbb Z}}

\def\part{\operatorname{part}}

\def\Prim{\operatorname{Prim}}
\def\col{\operatorname{col}}
\def\row{\operatorname{row}}
\def\End{{\operatorname{End}}}
\def\sign{{\operatorname{sign}}}
\def\Tab{{\operatorname{Tab}}}
\def\sTab{{\operatorname{sTab}}}
\def\sRow{{\operatorname{sRow}}}
\def\content{{\operatorname{content}}}
\def\word{{\operatorname{word}}}
\def\LT{{\operatorname{LT}}}

\def\eps{\epsilon}
\def\mf {\mathfrak}

\def\b{\mathfrak b}
\def\g{\mathfrak g}

\def\m{\mathfrak m}
\def\q{\mathfrak q}
\def\t{\mathfrak t}
\def\u{\mathfrak u}
\def\gl{\mathfrak{gl}}
\def\sl{\mathfrak{sl}}
\def\so{\mathfrak{so}}
\def\sp{\mathfrak{sp}}

\def\aa{\mathbf a}
\def\cc{\mathbf c}
\def\bp{{\mathbf p}}
\def\bq{{\mathbf q}}

\def\cS{\mathcal S}

\newdimen\Hoogte    \Hoogte=12pt    
\newdimen\Breedte   \Breedte=12pt   
\newdimen\Dikte     \Dikte=0.5pt    

\newenvironment{Young}{\begingroup
       \def\vr{\vrule height0.8\Hoogte width\Dikte depth 0.2\Hoogte}
       \def\fbox##1{\vbox{\offinterlineskip
                    \hrule height\Dikte
                    \hbox to \Breedte{\vr\hfill##1\hfill\vr}
                    \hrule height\Dikte}}
       \vbox\bgroup \offinterlineskip \tabskip=-\Dikte \lineskip=-\Dikte
            \halign\bgroup &\fbox{##\unskip}\unskip  \crcr }
       {\egroup\egroup\endgroup}

\def\Diagram#1{\relax\ifmmode\vcenter{\,\begin{Young}#1\end{Young}\,}\else%
              $\vcenter{\,\begin{Young}#1\end{Young}\,}$\fi}

\title{
Representation theory of type B and C standard Levi $W$-algebras
}
\author
{Jonathan Brown and Simon M. Goodwin}

\address{Mathematics Department, Gonzaga University,
Spokane, WA 99258, USA}
\email{brownj3@gonzaga.edu}

\address{School of Mathematics, University of Birmingham, Birmingham, B15 2TT,~UK}
\email{s.m.goodwin@bham.ac.uk}


\begin{document}

\begin{abstract}
We classify the
finite dimensional irreducible
representations
with
integral central character
of finite $W$-algebras $U(\g,e)$ associated to
standard Levi
nilpotent orbits in classical Lie
algebras of types B and C.  This classification
is given explicitly in terms of the highest weight theory
for finite $W$-algebras.
\end{abstract}

\maketitle

\section{Introduction}

Let $e$ be a nilpotent element in the Lie algebra $\g$ of a reductive algebraic group
$G$ over $\C$.  The finite $W$-algebra
$U(\g,e)$ associated to the pair $(\g,e)$ is an associative algebra obtained from $U(\g)$ by a certain
quantum Hamiltonian reduction.
There has been a great deal of recent interest in finite $W$-algebras and their representation theory,
for an overview see the survey article by Losev, \cite{Lo4}.

In recent work \cite{BG1} and \cite{BG3} the authors gave a combinatorial classification of the
finite dimensional irreducible
$U(\g,e)$-modules, where $\g$ is a
classical Lie algebra and $e$ is an even multiplicity nilpotent element;
we recall that $e$ is said to be {\em even multiplicity}
if all parts of the Jordan type of $e$ occur with even multiplicity.
This classification is given in terms of highest weight theory for finite $W$-algebras
from \cite{BGK}.

Now recall that a nilpotent element $e$ of $\g$ is said to be of {\em standard Levi type}
if $e$ is in the regular nilpotent orbit of a Levi subalgebra of $\g$.
It is easy to check that in case $\g$ is of classical type and $e$ is even multiplicity, then $e$ is
standard Levi.
In this paper we extend the results of \cite{BG1} to classify the
finite dimensional irreducible $U(\g,e)$-modules with integral central
character, where $\g$ is of type B or C
and $e$ is any {\em standard Levi nilpotent} element, see Theorem
\ref{T:mainintro}.
In \cite{BG4} we will supplement this theorem by classifying
finite dimensional irreducible $U(\g,e)$-modules for such $\g$ and $e$
of any (not necessarily integral) central character.  We recall, see for example
the footnote to \cite[Question 5.1]{Pr2},
that the centre of $U(\g,e)$ is canonically identified
with the centre of $U(\g)$, which allows one to define integral
central characters.

The situation for $\g$ of type D and $e$ standard Levi, but not
even multiplicity, is more awkward.  This can be dealt with 
using similar methods.

We remark here that finite $W$-algebras
corresponding to nilpotent elements of standard Levi type
are a natural class to consider.
This is because such finite $W$-algebras are particularly amenable to the highest
weight theory from \cite{BGK} as explained in \S\ref{ss:hw}.

\medskip

In \cite{LO}, Losev and Ostrik, have accomplished a classification of the finite dimensional $U(\g,e)$-modules
of integral central character for any reductive Lie algebra $\g$ in the following manner.
In \cite{Lo1} Losev gives a surjection from the primitive ideals
of finite codimension of $U(\g,e)$ to the primitive ideals of $U(\g)$ which have associated variety equal to
the closure $\overline{G \cdot e}$ of the $G$-orbit of $e$.
There is a natural action of the component group $C$ of the centralizer of $e$
in $G$ on the set of primitive ideals of $U(\g,e)$, as explained for
example in the introduction to \cite{Lo2}.  In \cite{Lo2} Losev extends
his results from \cite{Lo1} to show that the fibres of the above surjection
are precisely $C$-orbits.
Losev's and Ostrik's classification in \cite{LO} is accomplished by describing the fibres of this map, i.e.\
determining the stabilizer of the $C$-orbit for each fibre.
The primitive ideals with associated variety equal to $\overline{G \cdot e}$ can be described due to
methods of a variety of mathematicians in the 70's and 80's, see for example \cite{Ja}
and the references therein for details.

\medskip

We go on to explain the results of this paper in more detail, so we take $\g$
to be of type B or C, i.e.\ $\g = \so_{2n+1}$ or $\g = \sp_{2n}$ for some $n \in \Z_{\ge 2}$.
We recall that nilpotent orbits in $\g$ are parameterized by their Jordan type.  Thus
they are given by partitions of $2n+1$ (respectively $2n$) where
all even (respectively odd) parts occur with
even multiplicity when $\g = \so_{2n+1}$ (respectively $\g = \sp_{2n}$).
In this paper we consider only nilpotent orbits, which are standard Levi, but not even multiplicity, as the latter
are dealt with in \cite{BG1} and \cite{BG2}.
This means that the Jordan type of $e$ is given by a partition of the form
\[
  \bp = (p_1^{2a_1} < p_2^{2a_2} < \dots < p_{d-1}^{2a_{d-1}} < p_d^{2a_d+1} < p_{d+1}^{2a_{d+1}} < \dots < p_r^{2a_r}),
\]
i.e.\ all parts of $\bp$ occur with even multiplicity except for one part $p_d$, which
occurs with odd multiplicity.
It will be more convenient for us to re-index this partition and write it as
\[
  \bp = (p_1^2 \le p_2^2 \le \dots \le p_{d-1}^2 < p_0 \le p_d^2 \le \dots \le p_r^2).
\]

In this paper, we only consider finite dimensional irreducible representations for $U(\g,e)$
with integral central character.
As we explain in \S\ref{SS:Losev}, such representations occur only when $e$ is a special nilpotent element
in the sense of Lusztig from \cite[13.1.1]{Lu}.
In terms
of the partition $\bp$ this means that the dual partition of $\bp$ is the Jordan type of a
nilpotent orbit in $\g$.  Explicitly, this means that $p_i$ must be odd for all $i \ge d$ when $\g = \so_{2n+1}$,
or $p_i$ must be even for all $i \le d$ when $\g = \sp_{2n}$.  For the remainder of the
paper we assume that $\bp$ is a partition as above, which satisfies these conditions.

We use {\em symmetric pyramids} to describe much of the combinatorics underlying
$U(\g,e)$-modules.
The symmetric pyramid for $\bp$, denoted by $P$, is a finite connected collection of boxes in the plane such
that:
\begin{itemize}
\item[-]
the boxes are arranged in connected rows;
\item[-]
the boxes are symmetric with respect to both the $y$-axis and the $x$-axis.
\item[-]
each box is 2 units by 2 units;
\item[-]
the lengths of the rows from top to bottom are given by
\[
  p_1 \dots, p_r, p_0, p_r, \dots, p_1.
\]
\end{itemize}

An {\em s-table} with underlying symmetric pyramid $P$
is a skew symmetric (with respect to the origin) filling of $P$ with complex numbers.
We define $\sTab(P)$ to be a certain set of s-tables depending on whether
$\g = \so_{2n+1}$ of $\sp_{2n}$.  For $\g = \sp_{2n}$ we let $\sTab(P)$
denote the set s-tables with underlying symmetric pyramid $P$ such that
all entries are integers; whereas for $\g = \so_{2n+1}$, we define $\sTab(P)$
to be the s-tables such that either all entries are in $\Z$ or all entries are in
$\frac{1}{2} + \Z$.
Let $\sTab^\leq(P)$ denote the elements of $\sTab(P)$ which have non-decreasing rows.
As explained in \S\ref{SS:stables} the elements of $\sTab^\leq(P)$ parameterize the irreducible
highest weight $U(\g,e)$-modules; given $A \in \sTab(P)$ we write $L(A)$ for the corresponding
irreducible highest weight $U(\g,e)$-module.

An example of an s-table in $\sTab^{\le}(P)$, when $\g = \sp_{2n}$, $\bp = (5^2,4,2^2)$
and $P$ is the symmetric pyramid for $\bp$, is
\begin{equation} \label{EQ:stable}
\begin{array}{c}
\begin{picture}(100,100)
   \put(30,0){\line(1,0){40}}
   \put(0,20){\line(1,0){100}}
   \put(0,40){\line(1,0){100}}
   \put(0,60){\line(1,0){100}}
   \put(0,80){\line(1,0){100}}
   \put(30,100){\line(1,0){40}}

    \put(30,0){\line(0,1){20}}
    \put(50,0){\line(0,1){20}}
    \put(70,0){\line(0,1){20}}

   \put(0,20){\line(0,1){20}}
   \put(20,20){\line(0,1){20}}
   \put(40,20){\line(0,1){20}}
   \put(60,20){\line(0,1){20}}
   \put(80,20){\line(0,1){20}}
   \put(100,20){\line(0,1){20}}

   \put(10,40){\line(0,1){20}}
   \put(30,40){\line(0,1){20}}
   \put(50,40){\line(0,1){20}}
   \put(70,40){\line(0,1){20}}
   \put(90,40){\line(0,1){20}}

   \put(0,60){\line(0,1){20}}
   \put(20,60){\line(0,1){20}}
   \put(40,60){\line(0,1){20}}
   \put(60,60){\line(0,1){20}}
   \put(80,60){\line(0,1){20}}
   \put(100,60){\line(0,1){20}}

    \put(30,80){\line(0,1){20}}
    \put(50,80){\line(0,1){20}}
    \put(70,80){\line(0,1){20}}

     \put(38,10){\makebox(0,0){{-$7$}}}
     \put(58,10){\makebox(0,0){{-$6$}}}

     \put(8,30){\makebox(0,0){{-$9$}}}
     \put(28,30){\makebox(0,0){{-$8$}}}
     \put(48,30){\makebox(0,0){{-$5$}}}
     \put(68,30){\makebox(0,0){{-$4$}}}
     \put(88,30){\makebox(0,0){{-$2$}}}

     \put(18,50){\makebox(0,0){{-$3$}}}
     \put(38,50){\makebox(0,0){{-$1$}}}
     \put(60,50){\makebox(0,0){{$1$}}}
     \put(80,50){\makebox(0,0){{$3$}}}

     \put(10,70){\makebox(0,0){{$2$}}}
     \put(30,70){\makebox(0,0){{$4$}}}
     \put(50,70){\makebox(0,0){{$5$}}}
     \put(70,70){\makebox(0,0){{$8$}}}
     \put(90,70){\makebox(0,0){{$9$}}}

     \put(40,90){\makebox(0,0){{$6$}}}
     \put(60,90){\makebox(0,0){{$7$}}}
\end{picture}
\end{array}.
\end{equation}

The {\em left justification} of an s-table is the diagram created by left-justifying all of the s-table's rows.
We say an s-table is {\em justified row equivalent to column strict} if the row equivalence class of
its left justification contains a table in which every column is strictly decreasing; we note that there can be a
gap in the middle of some columns and we require entries to be strictly decreasing across this gap.
We write $\sTab^\cc(P)$ for the set of all $A \in \sTab(P)$, which are justified row equivalent to column strict.
It is easy to see that the example of the s-table above is an element of $\sTab^\cc(P)$.

Recall that $C$ denotes the component group of the centralizer of $e$ in $G$.
In \S\ref{SS:compgroup} we define an action of $C$ on the subset of $\sTab^\le(P)$
corresponding to finite dimensional $U(\g,e)$-modules.

Now we can state the main theorem of this paper.

\begin{Theorem} \label{T:mainintro}
Let $\g = \so_{2n+1}$ or $\sp_{2n}$, let $\bp$ be a partition corresponding to a standard Levi special nilpotent orbit in $\g$, let $e$ be an element of this orbit and let $P$ be
the symmetric pyramid for $\bp$.
Then
$$
\{ L(A) \mid  \text{$A \in \sTab^\leq(P)$, $A$ is $C$-conjugate to some $B \in \sTab^\cc(P)$}\}
$$
is a complete set of isomorphism classes of finite dimensional irreducible $U(\g,e)$-modules with integral central character.
Moreover, the $C$-action on s-tables agrees with the $C$-action on finite dimensional irreducible $U(\g,e)$-modules.
\end{Theorem}

Analogous results to \cite[Corollaries 5.17 and 5.18]{BG1} hold
in the present situation.
So when all parts of $\bp$ have the same parity, if $L(A)$ is finite dimensional, then in fact,
$A$ is row equivalent to column strict as an
s-table.  Thus in this case $L(A)$ can be obtained as a subquotient of
the restriction of a finite dimensional $U(\g(0))$-module via the Miura map.  We refer the reader to
the discussion before \cite[Corollary 5.18]{BG1} for more details, and to \S\ref{SS:def} below
for the definition of $\g(0)$.

Theorem \ref{T:mainintro} and the correspondence of finite dimensional irreducible $U(\g,e)$-modules and primitive
ideals of $U(\g)$ with associated variety $\overline{G \cdot e}$ discussed above,
allow us to deduce the following corollary.  It gives an explicit description of
the primitive ideals of $U(\g)$ which have
associated variety equal to $\overline{G \cdot e}$ and integral central character.
A method to classify these primitive ideals was originally given by Barbasch and Vogan in \cite{BV1}.
In the corollary, $L(\lambda_A)$ denotes the irreducible highest weight $U(\g)$-module
defined from an s-table $A$ as explained in \S\ref{SS:stables} below.

\begin{Corollary}
The set of primitive ideals with integral central character and associated variety $\overline{G \cdot e}$
is equal to
\[
    \{ \Ann_{U(\g)} L(\lambda_A) \mid A \in \sTab^\cc(P) \cap \sTab^\leq(P) \}.
\]
\end{Corollary}

Below we give an outline of the proof of Theorem \ref{T:mainintro}.

The key step is to deal with the case where $\bp$ has three parts.
We deal with this case using the relationship between
finite dimensional irreducible representations of $U(\g,e)$
and primitive ideals of $U(\g)$ with associated variety equal to
$\overline{G \cdot e}$.  Using this and results of Barbasch--Vogan
and Garfinkle on primitive ideals, we are able to classify
finite dimensional irreducible modules for $U(\g,e)$ and
explicitly describe the component group action.  These results are
stated in Theorems \ref{T:caction1} and \ref{T:caction2}.

In Section \ref{S:general}, we use inductive methods to deduce
Theorem \ref{T:mainintro}.  The important ingredients here are
``Levi subalgebras'' of $U(\g,e)$ as defined in \cite[\S3]{BG1}
and changing highest weight theories.  The latter is dealt with
in \cite{BG2} for the case of even multiplicity nilpotent orbit,
and we observe here that there is an analogous theory
in the present situation, see Proposition \ref{P:changehw}.

We note that if we were able to deal with the case where $\bp$ has three
parts by another means, for example from an explicit presentation
of the finite $W$-algebras, then we would be able to remove the
dependence on the results of Losev, Barbasch--Vogan and Garfinkle.
It would, therefore, be interesting and useful to have a presentation
of such finite $W$-algebras.

\subsection*{Acknowledgments}  This research is funded by EPSRC grant EP/G020809/1.

\section{Overview of finite $W$-algebras}

\subsection{Definition of the finite $W$-algebra $U(\mf{g},e)$} \label{SS:def}

Let $G$ be a reductive algebraic group over $\C$ with Lie algebra $\g$.
The finite $W$-algebra $U(\g,e)$ is defined in terms of a nilpotent
element $e \in \mf{g}$.  By the Jacobson--Morozov Theorem, $e$ embeds
into an $\sl_2$-triple $(e,h,f)$.  The $\ad h$ eigenspace
decomposition gives a grading on
$\mf{g}$:
\begin{equation} \label{EQ:grading}
  \mf{g} = \bigoplus_{j \in \Z} \mf{g}(j),
\end{equation}
where $\mf{g}(j) = \{ x \in \mf{g} \mid [h,x] = j x \}$.
Define the character $\chi : \mf{g} \to \C$ by
$\chi(x) = (x,e)$, where $(\,\cdot,\cdot)$ is a non-degenerate symmetric
invariant bilinear form on $\mf{g}$.
Then we can define a non-degenerate symplectic form
$\lan \,\cdot, \cdot \ran$
on $\mf{g}(-1)$ by
$\lan x,y \ran = \chi([y,x])$.
Choose a Lagrangian subspace
$\mf{l} \subseteq \mf{g}(-1)$ with respect to $\lan \,\cdot, \cdot \ran$,
and let
$\mf{m} = \mf{l} \oplus \bigoplus_{j \le -2} \mf{g}(j)$.
Let $\m_\chi = \{m - \chi(m) \mid m \in \mf{m}\}$.
The adjoint action of $\mf{m}$ on $U(\mf{g})$
leaves the left ideal $U(\g) \m_\chi$ invariant,
so there is an induced adjoint action of $\mf{m}$ on
$Q_\chi = U(\mf{g})/U(\g)\m_\chi$.
The space of fixed points $Q_\chi^{\mf{m}}$ inherits a well defined multiplication
from $U(\mf{g})$, making it an associative algebra, and we define the finite $W$-algebra
to be
$$
U(\mf{g},e) = Q_\chi^{\mf{m}} = \{ u + U(\g)\m_\chi \in Q_\chi \mid [x,u] \in U(\g)\m_\chi \text{ for all } x \in \m\}.
$$

We also recall here that the centre $Z(\g)$ of $U(\g)$ maps into $U(\g,e)$ via the inclusion
$Z(\g) \subseteq U(\g)$.  Moreover, it is known that this defines an isomorphism
between the $Z(\g)$ and the centre of $U(\g,e)$, see the footnote to \cite[Question 5.1]{Pr2}.
We use this isomorphism to identify the centre of $U(\g,e)$ with $Z(\g)$, which in particular
allows us to define integral central characters for $U(\g,e)$-modules.

\begin{Remark}
We do not really require the definition of the finite $W$-algebra in this paper, but include
it for completeness.  Also we note here that there are different equivalent definitions
of the finite $W$-algebra in the literature.  Above we have given the Whittaker model
definition, as it is the shortest and most convenient for our purposes here.
\end{Remark}

\subsection{Skryabin's equivalence and Losev's map of primitive ideals} \label{SS:Losev}
The left $U(\g)$-module $Q_\chi$ is also a right $U(\g,e)$-module,
so there is a functor
\[
\cS : U(\g,e)\text{-mod} \to U(\g)\text{-mod}, \quad
M \mapsto Q_\chi \otimes_{U(\g,e)} M
\]
where $M$ is a $U(\g,e)$-module.
In \cite{Sk} Skryabin shows that
$\cS$ is a equivalence of categories between
$U(\g,e)$-mod and the category of {\em Whittaker modules for $e$}, i.e.\ the category of $U(\g)$-modules on which
$\m_\chi$ acts locally nilpotently.

For an algebra $A$ let $\Prim A$ denote the set of primitive ideals of $A$.
In \cite{Lo2} Losev shows that there exists a map
\[
 \cdot ^\dagger : \Prim U (\g, e) \to \Prim U (\g),\quad I \mapsto I^\dagger
\]
with the following properties:
\begin{enumerate}
 \item
 $\cdot^\dagger$ preserves central characters, i.e.\ $I \cap Z(\g) = I^\dagger \cap Z(\g)$
 for any $I \in \Prim(U(\g,e)$, under the identification of the centre of $U(\g,e)$ with $Z(\g)$.
 \item
   $\cdot^\dagger$ behaves well with respect to Skryabin's equivalence in the sense that
\[
  \Ann_{U(\g)} \cS(M) = (\Ann_{U(\g,e)} M)^\dagger
\]
for every irreducible $U(\g,e)$-module $M$;
  \item
    the restriction of $\cdot ^\dagger$ to $\Prim_0 U(\g,e)$,
     the set of primitive ideals of $U(\g,e)$ of finite co-dimension,
      is a surjection onto $\Prim_e U(\g)$, the set of primitive ideals
      of $U(\g)$ with associated variety equal to $\overline {G \cdot e}$.
\item
the fibres of $\cdot^\dagger$ restricted to $\Prim_0 U(\g,e)$ are $C$-orbits, where $C$ is the component group of the
centralizer of $e$.  See, for example the introduction to \cite{Lo2} for an explanation of the action of $C$
on $\Prim_0 U(\g,e)$.
\end{enumerate}

\subsection{Highest weight theory and Losev's map} \label{ss:hw}
By using the highest weight theory for finite $W$-algebras developed by
Brundan, Kleshchev and the second author in \cite{BGK},
the map $\cdot ^\dagger$ from the previous subsection can be explicitly calculated
in terms of highest weight modules for $U(\g,e)$ and $U(\g)$.

The key part of this highest weight theory is the use of a minimal Levi
subalgebra $\g_0$ which contains $e$.  In \cite[Theorem 4.3]{BGK}
it is proved that there is a certain subquotient of $U(\g,e)$, which is isomorphic
to $U(\g_0,e)$.  Then in \cite[\S4.2]{BGK} it is explained how a choice of a parabolic subalgebra
$\q$ with Levi factor $\g_0$ leads
to a highest weight theory for $U(\g,e)$, in which $U(\g_0,e)$ plays the role of the Cartan subalgebra in
the usual highest weight theory for reductive Lie algebras.  This leads to a definition
of Verma modules for $U(\g,e)$ by ``parabolically inducing'' $U(\g_0,e)$-modules up to $U(\g,e)$-modules.
Then \cite[Theorem 4.5]{BGK} says that these Verma modules have irreducible heads, and that any finite dimensional
irreducible $U(\g,e)$-module is isomorphic to one of these irreducible heads.  This gives
a method to explicitly parameterize finite dimensional irreducible $U(\g,e)$-modules, though
a classification of $U(\g_0,e)$-modules in general is unknown at present.

When $e$ is of standard Levi type, the classification of $U(\g_0,e)$-modules is known.
By a theorem of Kostant in \cite{Ko} and the Harish-Chandra isomorphism, we have that
$U(\g_0,e) \cong Z(\g_0) \cong S(\t)^{W_0}$,
where $\t$ is a maximal toral subalgebra of $\g$ and $W_0$ is the Weyl group of $\g_0$.
Hence the finite
dimensional irreducible $U(\g_0,e)$-modules are all one dimensional, and they are parameterized by
the $W_0$-orbits on $\t^*$.
We choose $\t$ as specified in \cite[S5.1]{BGK},
and let $\Lambda \in \t^*/W_0$ be a $W_0$-orbit.
In \cite[\S5.1]{BGK} an explicit isomorphism $U(\g_0,e) \to S(\t)^{W_0}$ is given.
Using this isomorphism and our choice of $\q$ we let $M(\Lambda,\q)$ denote the Verma module for $U(\g,e)$
induced from $\Lambda$, and we write $L(\Lambda,\q)$ for the irreducible head of $M(\Lambda,\q)$.  We
note that there are ``shifts'' involved in the isomorphisms above and thus in the definition of $M(\Lambda,\q)$
as defined in \cite[Sections 4 and 5]{BGK}.

Let $\u$ be the nilradical of $\q$,
and let $\b_0$ be a Borel subalgebra of $\g_0$ which contains $\t$ so that $\b = \b_0 \oplus \u$ is a Borel subalgebra of
$\g$.
For $\lambda \in \t^*$
let $L(\lambda,\b)$ denote the highest weight irreducible $\g$-module defined in terms of $\b$, with highest weight
$\lambda - \rho$ (where $\rho$ is the half-sum of the positive roots for $\b$).

The theorem below allows us to explicitly calculate Losev's map
$\cdot ^\dagger$ on primitive ideals in terms of highest weight modules.
In \cite[\S5.1]{BGK}
it is shown that this theorem
follows from \cite[Theorem 5.1]{MS} and \cite[Conjecture 5.3]{BGK}.
Also \cite[Conjecture 5.3]{BGK}
was verified in \cite[Theorem 5.1.1]{Lo3},
except for a technical point which was resolved in \cite[Proposition 3.10]{BG1}.

\begin{Theorem} \label{T:BGKt1}
    Let $\Lambda \in \t^*/W_0$ and let $\lambda \in \Lambda$ be antidominant for $\b_0$.
    Then
   $$
   (\Ann_{U(\g,e)} L(\Lambda,\q))^\dagger  = \Ann_{U(\g)} L(\lambda,\b).
   $$
\end{Theorem}

One consequence of this theorem is that if $e$ is not a special nilpotent element then $U(\g,e)$ has no finite dimensional irreducible representations of
integral central character.  This is due to results of Barbasch and Vogan in \cite{BV1} and \cite{BV2}, which imply that the associated variety of
$\Ann_{U(\g)} L(\lambda,\b)$ is a special nilpotent orbit if and only if $\lambda$ is integral.

The following theorem is \cite[Conjecture 5.2]{BGK}, which, as is explained in \cite[\S5]{BGK}, follows from
\cite[Conjecture 5.3]{BGK}.

\begin{Theorem} \label{T:BGKt2}
Let $\Lambda \in \t^*/W_0$ and let $\lambda \in \Lambda$ be antidominant for $\b_0$.
Then $L(\Lambda,\q)$ is a finite dimensional if and only if the associated variety of $\Ann_{U(\g)} L(\lambda,\b)$ is equal to $\overline{G \cdot e}$.
\end{Theorem}

\section{Combinatorics of s-tables and finite $W$-algebras}

\subsection{Realizations of $\so_{2n+1}$ and $\sp_{2n}$}
\label{SS:realize}
In the case $\g = \so_{2n+1}$, we realize $\g$ in the following way.
Let $V = \C^{2n+1}$ have basis $\{e_1, \dots, e_n, e_0, e_{-n}, \dots, e_{-1}\}$.
Then we take $\gl_{2n+1} = \End(V)$ as having basis $\{e_{i,j} \mid i, j = 0, \pm 1, \dots, \pm n \}$.
We define the bilinear form $(\,\cdot,\cdot)$ on $V$ by declaring that $(e_i, e_j) = \delta_{i,-j}$.
Then we set
$$
\g = \so_{2n+1} = \{x \in \gl_{2n+1} \mid (xv,w) = -(v,xw) \text{ for all } v,w \in V\}.
$$
Note that $\g$ has basis $\{f_{i,j} \mid i, j = 0, \pm 1, \dots, \pm n,\, i +j > 0\}$, where
$f_{i,j} = e_{i,j} - e_{-j,-i}$.
We choose $\t = \{f_{i,i} \mid i = 1, \dots, n\}$ as a maximal toral subalgebra, so that
$\t^*$ has basis $\{\eps_i \mid i = 1, \dots, n\}$ where $\eps_i \in \t^*$ is defined via $\eps_i(f_{j,j}) = \delta_{i,j}$ for $i,j > 0$.  We write $\Phi$ for the root system of $\g$ with respect to $\t$.
Let $\b = \langle f_{i,j} \mid i \leq j \rangle$ be the Borel subalgebra of upper triangular matrices in $\g$.
Then the corresponding system of positive roots is given by
$\Phi^+ = \{\eps_i \pm \eps_j \mid 1 \leq i < j \leq n\} \cup \{\eps_i \mid i =1, \dots, n\}$.

For $\g = \sp_{2n}$, we
let $V = \C^{2n}$ have basis $\{e_1, \dots, e_n, e_{-n}, \dots, e_{-1}\}$.
Then we realize $\gl_{2n} = \End(V)$ as having basis $\{e_{i,j} \mid i, j = \pm 1, \dots, \pm n \}$.
We define the bilinear form $(\,\cdot,\cdot)$ on $V$ by declaring that $(e_i, e_j) = \sign(i) \delta_{i,-j}$ and set
$$
\g = \sp_{2n} = \{x \in \gl_{2n} \mid (xv,w) = -(v,xw) \text{ for all } v,w \in V\}.
$$
Then $\g$ has basis $\{f_{i,j} \mid i, j = \pm 1, \dots, \pm n,\, i +j \geq 0\}$, where
$f_{i,j} = e_{i,j} - \sign(i) \sign(j) e_{-j,-i}$.
We choose $\t = \{f_{i,i} \mid i = 1, \dots, n\}$ as a maximal toral subalgebra, so that
$\t^*$ has basis $\{\eps_i \mid i = 1, \dots, n\}$ where $\eps_i \in \t^*$ is defined via $\eps_i(f_{j,j}) = \delta_{i,j}$ for $i,j > 0$.  We write $\Phi$ for the root system of $\g$ with respect to $\t$.
We choose the Borel subalgebra $\b = \langle f_{i,j} \mid i \leq j \rangle$ of upper triangular
matrices in $\g$.
Then the corresponding system of positive roots is given by
$\Phi^+ = \{\eps_i \pm \eps_j \mid 1 \leq i < j \leq n\} \cup \{2 \eps_i \mid i =1, \dots, n\}$.

\subsection{Standard Levi nilpotent elements and symmetric pyramids} \label{SS:symmp}
Recall from the introduction that we are considering nilpotent orbits in $\g$, which
are special and standard Levi, but not even multiplicity.
The Jordan type for such a nilpotent orbit is of the form
\begin{equation} \label{e:bp}
\bp = (p_1^2 \le \dots \le p_{d-1}^2 < p_0 \le p_d^2 \le \dots \le p_r^2).
\end{equation}
Moreover, $p_i$ must be odd for all $i \ge d$ when $\g = \so_{2n+1}$,
or $p_i$ must be even for all $i < d$ when $\g = \sp_{2n}$.

Also recall, from the introduction, the definition of the symmetric pyramid $P$ for $\bp$.
We form a diagram $K$ called the {\em coordinate pyramid for $\bp$} by filling the boxes of $P$
with
$1, \dots, n, -n, \dots, -1$ if $\g = \sp_{2n}$, or with
$1, \dots, n, 0, -n, \dots, -1$ if $\g = \so_{2n}$,
across rows from top to bottom.
For example, for $\g = \sp_{18}$ and  $\bp = (5^2,4,2^2)$, we have
\[
K =
\begin{array}{c}
\begin{picture}(100,100)
   \put(30,0){\line(1,0){40}}
   \put(0,20){\line(1,0){100}}
   \put(0,40){\line(1,0){100}}
   \put(0,60){\line(1,0){100}}
   \put(0,80){\line(1,0){100}}
   \put(30,100){\line(1,0){40}}

    \put(30,0){\line(0,1){20}}
    \put(50,0){\line(0,1){20}}
    \put(70,0){\line(0,1){20}}

   \put(0,20){\line(0,1){20}}
   \put(20,20){\line(0,1){20}}
   \put(40,20){\line(0,1){20}}
   \put(60,20){\line(0,1){20}}
   \put(80,20){\line(0,1){20}}
   \put(100,20){\line(0,1){20}}

   \put(10,40){\line(0,1){20}}
   \put(30,40){\line(0,1){20}}
   \put(50,40){\line(0,1){20}}
   \put(70,40){\line(0,1){20}}
   \put(90,40){\line(0,1){20}}

   \put(0,60){\line(0,1){20}}
   \put(20,60){\line(0,1){20}}
   \put(40,60){\line(0,1){20}}
   \put(60,60){\line(0,1){20}}
   \put(80,60){\line(0,1){20}}
   \put(100,60){\line(0,1){20}}

    \put(30,80){\line(0,1){20}}
    \put(50,80){\line(0,1){20}}
    \put(70,80){\line(0,1){20}}

     \put(38,10){\makebox(0,0){{-2}}}
     \put(58,10){\makebox(0,0){{-1}}}

     \put(8,30){\makebox(0,0){{-7}}}
     \put(28,30){\makebox(0,0){{-6}}}
     \put(48,30){\makebox(0,0){{-5}}}
     \put(68,30){\makebox(0,0){{-4}}}
     \put(88,30){\makebox(0,0){{-3}}}

     \put(20,50){\makebox(0,0){{8}}}
     \put(40,50){\makebox(0,0){{9}}}
     \put(58,50){\makebox(0,0){{-9}}}
     \put(78,50){\makebox(0,0){{-8}}}

     \put(10,70){\makebox(0,0){{3}}}
     \put(30,70){\makebox(0,0){{4}}}
     \put(50,70){\makebox(0,0){{5}}}
     \put(70,70){\makebox(0,0){{6}}}
     \put(90,70){\makebox(0,0){{7}}}

     \put(40,90){\makebox(0,0){{1}}}
     \put(60,90){\makebox(0,0){{2}}}
\end{picture}
\end{array}.
\]
We let $\col(i)$ denote the $x$-coordinate of the centre of the box of $K$ which contains $i$.
However,
we use $\row(i)$ to denote the row of $K$ which contains $i$ when we label the rows of $K$ by
$1,\dots,r,0,-r,\dots,-1$,
from top to bottom; so that $p_i$ is the length of row $i$.

We define $e \in \g $ by
\begin{equation} \label{EQ:e}
e = \sum_{i,j}
f_{i,j},
\end{equation}
where the sum is over all adjacent pairs
$\Diagram{ $\scriptstyle i$ & $\scriptstyle j$ \cr}$ in $K$, so that
$e$ is in the nilpotent $G$-orbit with Jordan type $\bp$.

We also use $K$ to conveniently define many of the objects required for the
definition of $U(\g,e)$ and the highest weight theory.

Let $h = \sum_{i=1}^n - \col(i) f_{i,i}$,
then $(e,h,f)$ is an $\sl_2$-triple for some $f \in \g$.
Furthermore
the grading from \eqref{EQ:grading} on $\g$ is given by
$\g(k) =
\langle f_{i,j} \mid \col(j) - \col(i) = k \rangle$.  Then we have that
$\m = \langle f_{i,j} \mid \col(i) > \col(j) \rangle$,
and we use these to form the finite $W$-algebra $U(\g,e)$ as in \S\ref{SS:def}.

We take
$\g_0 = \langle f_{i,j} \mid \row(i) = \row(j) \rangle$.
So $\g_0$ is a minimal Levi subalgebra which contains $e$, and $e$ is a regular nilpotent element of $\g_0$.
In the case $\g = \so_{2n+1}$, we have
$$
\g_0 \cong \so_{p_0} \oplus
\bigoplus_{i=1}^r
\gl_{p_i};
$$
and in the case $\g = \sp_{2n}$, we have
$$
\g_0 \cong \sp_{p_0} \oplus
\bigoplus_{i=1}^r
\gl_{p_i}.
$$
We choose $\q = \langle f_{i,j} \mid \text{the row containing $i$ is above or equal to the row containing $j$} \rangle$.
Then $\q$ is a parabolic subalgebra of $\g$ with Levi factor $\g_0$.
Let $\b_0 = \b \cap \g_0$, so that
$\b_0$ is a Borel subalgebra of $\g_0$ which satisfies
$\b = \b_0 \oplus \mathfrak{u}$, where $\mathfrak{u}$ is the nilradical of $\q$.

\subsection{Tables and s-tables} \label{SS:stables}
We use the definitions and notation regarding frames, tables, s-frames
and s-tables from \cite[\S4]{BG1}.  Below we explain how these are used
to label highest weight modules for $U(\g,e)$.

For this purpose we let
$W_r$ be the Weyl group of type $\mathrm B_r$, which acts
on $\{0,\pm 1,\dots.\pm r\}$ in the natural way.  We write
$\bar s_i = (i, i+1)(-i, -i-1)$ for $i=1,\dots,r-1$ for the
standard generators of $W_r$.  Let $\bar S_r$ be the subgroup of $W_r$
generated by $\bar s_i$ for $i=1,\dots,r-1$.

Given $\sigma \in W_r$, we define $\sigma \cdot P$ to be the diagram obtained
from $P$ by permuting rows according to $\sigma$, so that $\sigma \cdot P$ is an $s$-frame.
We recall that by an {\em s-table with frame $\sigma \cdot P$} we mean a skew symmetric
(with respect to the origin) filling of $\sigma \cdot P$ with complex numbers.
Then we define $\sTab(\sigma \cdot P)$ to be the set of s-tables with frame $\sigma \cdot P$  such
that: all entries are integers if $\g = \sp_{2n}$; and either all entries are in $\Z$ or all entries are in
$\frac{1}{2} + \Z$ if $\g = \so_{2n+1}$.

We let $\sigma \cdot K$ be the s-table obtained from $K$ by permuting
rows according to $\sigma$.  Now given $A \in \sTab(\sigma \cdot P)$
we define $\lambda_A = \sum_{i=1}^n a_i \eps_i$ where
$a_i$ is the entry of $A$ in the same box as $i$ in $\sigma \cdot K$.  In this way we get an identification of
$\sTab(\sigma \cdot P)$ with the set of integral weights in $\t^*$; we write $\t^*_\Z$ for the set of integral
weights of $\t$.

The {\em row equivalence class} of an s-table is the set of s-tables which can be created by permuting entries within rows.  We let $\sRow( \sigma \cdot P)$ denote the set of row equivalence classes of $\sTab(\sigma \cdot P)$.
Then $\sRow(\sigma \cdot P)$ identifies naturally with $\t^*_\Z/W_0$, where $W_0$ is the Weyl group of $\g_0$.
Let $\sTab^\le(\sigma \cdot P)$ denote the elements of  $\sTab(\sigma \cdot P)$ which have non-decreasing rows.
Then every element of that $\sRow(\sigma \cdot P)$ contains a unique element of $\sTab^\le(\sigma \cdot P)$.

We label the rows
of $\sigma \cdot K$ with $1,\dots,r,0,-r,\dots,-1$ from top to bottom.
Now we define $\q_\sigma$ to be generated by the by $f_{ij}$ for which the row of $\sigma \cdot K$ in which $i$ appears is
above of equal to the row containing $j$.
Then $\q_\sigma$ is parabolic subalgebra of $\g$ with Levi factor $\g_0$,
so we can use it to define the irreducible highest weight modules $L(\Lambda,\q_\sigma)$, for
$\Lambda \in \t^*/W_0$  as defined in \S\ref{ss:hw}.

Given $\Lambda \in \t^*_\Z/W_0$, there is a unique $A \in \sTab^{\le}(\sigma \cdot P)$ whose row
equivalence class $\bar A \in \sRow(P)$ is identified with $\Lambda$ as above.
We let $L_\sigma(A)$ denote $L(\Lambda,\q_\sigma)$.

Let $\u_\sigma$ be the nilradical of $\q_\sigma$, and define $\b_\sigma =
\b_0 \oplus \u_\sigma$, which is a Borel subalgebra of $\g$.  We write
$L_\sigma(\lambda_A)$ for the irreducible highest weight $U(\g)$ module,
with respect to $\b_\sigma$,
with highest weight $\lambda_A - \rho_\sigma$, where $\rho_\sigma$
is the half sum of positive roots for $\b_\sigma$.

Now Theorem \ref{T:BGKt1} can be restated in our present notation as follows.

\begin{Theorem} \label{T:BGKt3}
   Let $\sigma \in W_r$ and $A \in \sTab^\leq(P)$.  Then
  $(\Ann_{U(\g,e)} L_\sigma(A))^\dagger = \Ann_{U(\g)} L_\sigma(\lambda_{A})$.
\end{Theorem}

We are mainly interested in the case where $\sigma = 1$.
Here we have $\q_\sigma = \q$, and we
write $L(A)$ instead of $L_1(A)$ and
$L(\lambda_{A})$ instead of  $L_1(\lambda_A)$ for $A \in \sTab(P)$.

Thanks to Theorem \ref{T:BGKt3} our aim to classify the finite
dimensional irreducible $U(\g,e)$-modules
and understand the component group action on these modules
can be broken down to answering the following questions:
\begin{enumerate}
 \item
    For which $A \in \sTab^\leq(P)$ is the associated variety of
$\Ann_{U(\g)} L(\lambda_A)$ equal to $\overline{G \cdot e}$?
\item
   Given $A \in \sTab^\leq(P)$ such that $L(A)$ is finite dimensional, which $B \in \sTab^\leq(P)$ satisfy
$\Ann_{U(\g)} L(\lambda_A) = \Ann_{U(\g)} L(\lambda_B)$?
\end{enumerate}
In the case that $\bp$ has 3 parts we answer these two questions in Sections \ref{S:Ccase} and \ref{S:Bcase}.
The key ingredients in answering the first question are the Robinson--Schensted and Barbasch--Vogan algorithms explained in \S\ref{ss:RS}
and \S\ref{ss:BV}.
For the second question we use Vogan's $\tau$-equivalence on integral weights of $\g$, which is explained in \S\ref{ss:tau}.

In moving from the 3 row case to the general case, a key role is played by the different choices
of highest weight theories determined by the different parabolic subalgebras $\q_\sigma$ for $\sigma
\in W_r$.  This dependence follows easily from the results for the case of even multiplicity nilpotent
elements established in \cite{BG2}, which hold in the present
situation, the key result for us is Proposition~\ref{P:changehw}.  We also require the explicit description
of the action of the component group
on the set of finite dimensional irreducible $U(\g,e)$-modules in terms of s-tables, which is given in Proposition \ref{P:caction3}.
The proof of Theorem \ref{T:mainintro}
for the general case is then dealt with in \S\ref{ss:generalproof}.

\subsection{The component group} \label{ss:comp}

Recall that $C$ denotes the component group of the centralizer
of $e$ in $G$.  Here we take $G$ to be the adjoint group of $\g$, so
$G$ is either $\mathrm{SO}_{2n+1}$ or $\mathrm{PSp}_{2n}$.

A specific realization of $C$ is given as follows.
Let $0< p_{i_1} < \dots < p_{i_s}$ be the maximal distinct parts of $\bp$
such that $p_{i_j} \ne p_0$ and $p_{i_j}$ is
odd (respectively even) when $\g = \so_{2n+1}$ (respectively $\sp_{2n}$);
by maximal we mean that if $p_k = p_{i_j}$, then $k \le i_j$.
Define the matrices
$c_1,\dots,c_s$  corresponding to
$p_{i_1},\dots,p_{i_s}$
for $p_{i_k} \ne p_0$
by setting
\[
\dot c_k = \sum_{\substack{
         -n \le i,j \le n \\ \col(i) = \col(j) \\ \row(i) = i_k \\ \row(j) = -i_k}}
    \sign(\col(i)) (e_{i,j} + e_{j,i})
    + \sum_{\substack{ -n \le i \le n \\ \row(i) \neq \pm i_k}} e_{i,i}.
\]
Then one can calculate that $\dot c_k$ centralizes of $e$.
Furthermore the argument used in \cite[Section 6]{Bro2}
can be adapted to show that their images
$c_1, \dots, c_s$ in $C$ generate $C \iso \Z_2^s$.

As mentioned in \S\ref{SS:Losev} there is an action
of $C$ on $\Prim U(\g,e)$, and thus on isomorphism
classes of irreducible modules, and as explained in \cite[\S2.3]{BG1}
this can be seen as ``twisting'' modules by elements of $C$ (up to
isomorphism).
Given an irreducible $U(\g,e)$-module $L$ and $b \in C$, we write
$b \cdot L$ for the twisted module; we note that this is a minor abuse of notation
as $b \cdot L$ is only defined up to isomorphism.

\section{Some combinatorics for s-tables}

\subsection{The Robinson--Schensted Algorithm} \label{ss:RS}
We use the formulation of the
Robinson--Schensted algorithm from \cite[\S4]{BG1}.  We denote the
Robinson--Schensted algorithm by $\RS$ and recall that it
takes as input a word of integers (or more generally complex numbers) or a table
and outputs a tableau.

There are two lemmas about the Robinson--Schensted algorithm that we use repeatedly in the sequel
we state them below for convenience; they can be found in \cite[\S3]{Fu}.  For a word $w$, we define $\ell(w,k)$ to be the maximum possible sum of the
lengths of $k$ disjoint weakly increasing subsequences of $w$, and \textctc$(w,k)$ to be the maximum possible sum of the lengths of
$k$ disjoint strictly decreasing subsequences of $w$.  We write $\part(T)$ to denote the partition underlying a tableau $T$.

\begin{Lemma} \label{L:altRS}
Let $w$ be a word of integers and let $\bq = (q_1 \ge
\dots \ge q_n) = \part(\RS(w))$. Then for all $k \ge 1$,
$\ell(w,k) = q_1 + \dots + q_k$.
\end{Lemma}

\begin{Lemma} \label{L:altcRS}
Let $w$ be a word of integers and let $\bq^T = (q^*_1 \ge \dots \ge
q^*_n)$ be the dual partition to $\bq = \part(\RS(w))$. Then for all
$k \ge 1$, {\em \textctc}$(w,k) = q^*_1 + \dots + q^*_k$.
\end{Lemma}

An elementary fact about the Robinson--Schensted algorithm required later is stated in Lemma \ref{L:largersmaller} below; it
is easily deduced from
Lemma \ref{L:altRS}.
Suppose $u, w$ are words of integers and $a,b$ are integers such that
$a > b$, then we say the transposition of the word $u a b w$ to $u b a w$ is
a {\em larger-smaller transposition}.  Also we refer the reader to \cite[\S2]{Fu}
for the definition of Knuth equivalences.

\begin{Lemma} \label{L:largersmaller}
If $u$ and $w$ are words of integers and $w$ can be obtained from $u$
by a sequence of Knuth equivalences and larger-smaller transpositions then
$\part(\RS(u)) \leq \part(\RS(w))$.
\end{Lemma}

The following theorem extends \cite[Theorem 4.6]{BG1} and is
important for us later.  In the statement $P$ is the
symmetric pyramid for the partition $\bp$ as in the previous section.

\begin{Theorem} \label{T:recs}
Let $A,B \in \sTab^\leq(P)$.  Then:
\begin{enumerate}
\item[(i)]
$A$ is justified row equivalent to column
strict if and only if $\part(\RS(A)) = \bp$.
\item[(ii)]
If $\part(\RS(A)) = \bp$, then $\RS(A) = \RS(B)$ if and only if $A=B$.
\end{enumerate}
\end{Theorem}

\begin{proof}
Part (i) can be proved in the same way as \cite[Theorem 4.6]{BG1}.
we just need to check the proof still holds
if $A$ has an odd number of rows and the middle row of $A$ is not $A$'s longest row.
The only thing to
check is that there is a sequence of row swaps which transform $A$ into a tableau, such that
the convexity conditions
required by \cite[Lemma 4.9]{BG1} are satisfied, which is clear.

To prove (ii), we simply note that each row swap from the sequence of row swaps from (i) which turns $A$ into a tableau
is invertible.
\end{proof}

Lastly in this section we give the following theorem, which is important later on.

\begin{Theorem} \label{T:recs2}
Let $A,B \in \sTab^\cc(P)$.
Suppose that $\Ann_{U(\g)} L(\lambda_A) = \Ann_{U(\g)} L(\lambda_B)$.
Then $A=B$.
\end{Theorem}

\begin{proof}
As $A$ and $B$ are justified row equivalent to column strict then by
Theorem \ref{T:recs} $\part(\RS(A)) = \part(\RS(B)) = \bp$.
Using \cite[Theorem 3.5.11]{Ga3} and \cite[Proposition 4.2.3]{Le}, and
that $\Ann_{U(\g)} L(\lambda_A) = \Ann_{U(\g)} L(\lambda_B)$
we able to deduce that $\RS(A) = \RS(B)$.  Now the statement
follows from Theorem \ref{T:recs}.
\end{proof}

\subsection{Row swapping} \label{ss:rowswap}

In the proof of Theorem~\ref{T:recs} above we have mentioned the row swapping operations $s_i \star$ on tables
as defined in \cite[\S4]{BG1} and \cite[\S4]{BG2}.
An important ingredient for the definition of these row swapping operations is the notion of best fitting as
defined in \cite[\S4]{BG1}, which we use repeatedly in the sequel.

We also require the operations $\bar s_i \star$ for s-tables and we use the notation from \cite[\S5]{BG2}.
Recall that for $\sigma \in W_r$ and an s-table $A \in \sTab^\le(\sigma \cdot P)$, either $\bar s_i \star A$ is undefined
or it is an element of $\sTab^\le(\bar s_i\sigma \cdot P)$.
These operations can be extended to operations by elements of $\bar S_r$; the proof of \cite[Proposition 5.5(i)]{BG2}
goes through in our situation to show that this is well defined.

The following proposition is a version of \cite[Proposition 5.3(ii)]{BG2} in the present setting
and its proof adapts immediately.

\begin{Proposition} \label{P:changehw}
Let $\sigma \in W_r$, $\tau \in \bar S_r$ and $A \in \sTab^\le(\sigma \cdot P)$
Suppose that $\tau \star A$ is defined.
Then $L_\sigma(A) \cong L_{\tau \sigma}(\tau \star A)$.
\end{Proposition}

Also we state the following lemma as it is key for the proof of Theorem \ref{T:mainintro}.
It is \cite[Lemma 5.11]{BG1}, adapted to our situation and the same proof holds.  In the statement
$A^1_r$ denotes the table formed by rows $1$ to $r$ of $A$.

\begin{Lemma} \label{L:Aplus}
For $A \in \sTab^\leq(P)$ suppose that $L(A)$ is finite dimensional, and let
$\tau \in \bar S_r$.
Then $A^1_r$ is justified row equivalent to column strict and $\tau \star A$ is defined.
\end{Lemma}

\subsection{The Barbasch--Vogan algorithm} \label{ss:BV}

The Barbasch--Vogan algorithm from \cite{BV1}
takes as input $\lambda$, an integral weight for a classical Lie algebra of type B or C, and outputs
$\operatorname{BV}(\lambda)$, the Jordan type of the associated variety of $\Ann_{U(\g)}L(\lambda)$.
Below we recall the description of it given in \cite[\S5.2]{BG1}.
We note that there is a version of it for type D, but we do not require that here.

We need to define the {\em content of a partition}.
Let $\bq = (q_1 \le q_2 \le \dots, \le q_m)$ be a partition.
By inserting $0$ at the beginning if necessary, we may assume that $m$ is odd.
Let $(s_1, \dots, s_k)$, $(t_1, \dots, t_l)$ be such that
as unordered lists,
$(q_1, q_2+1, q_3+2, \dots, q_r+r-1)$ is equal to
$(2 s_1, \dots, 2 s_k, 2 t_1 +1, \dots, 2 t_l+1)$.
Now we define the content of $\bq$ to be the unordered list
\[
  \content(\bq) = (s_1, \dots, s_k, t_1, \dots, t_l).
\]

\noindent {\bf Algorithm:}

\noindent {\em Input:} $\lambda = \sum_{i=1}^n a_i \eps_i$, an integral weight in $\t^*$.

\noindent {\em Step 1:} Calculate $\bq =
\part(\RS(a_1 , \dots , a_n , -a_n , \dots , -a_1 ))$.

\noindent {\em Step 2:}
Calculate $\content(\bq)$. \\
Let $(u_1 \leq \dots \leq u_{2k+1})$ be the sorted list with the same entries as $\content(\bq)$. \\
For $i=1,\dots, k+1$ let $s_i = u_{2 i-1}$. \\
For $i=1,\dots, k$ let $t_i = u_{2i}$.

\noindent {\em Step 3:}
Form the list $(2 s_1+1, \dots, 2 s_{k+1}+1, 2 t_1,\dots, 2 t_{k})$. \\
In either case let $(v_1< \dots < v_k)$ be this list after sorting.

\noindent {\em Output:} $\BV(\lambda) = \bq' = (v_1, v_2-1, \dots, v_{2k+1}-2k)$.

\medskip

We note that the output partition $\bq'$ is a partition (potentially with an extraneous zero
at the beginning) is the Jordan type of a {\em special} nilpotent orbit of $\g$; this proved
in \cite{BV1}.

For our purposes in this paper we also need
a modified version of the algorithm to use in the case
$\g = \so_{2n+1}$.   This modified version is denoted by $\BV'$.
It works in exactly the same way as $\BV$ except that instead of calculating
$\RS(a_1 , \dots , a_n , -a_n , \dots , -a_1 )$, we calculate
$\RS(a_1 , \dots , a_n , 0, -a_n , \dots , -a_1 )$ in Step 1.

In Corollary \ref{C:BV} in the appendix to this paper it is proved that
$$
\BV(\lambda) = \BV'(\lambda)
$$
for $\lambda \in \t^*$ in the case $\g = \so_{2n+1}$.
This proof of this is entirely combinatorial and may be of independent
interest so it is has been placed in an appendix.  In light of this
we redefine $\BV(\lambda)$, so that is the old $\BV(\lambda)$ in the case
$\g = \sp_{2n}$ and is $\BV'(\lambda)$ in the case
$\g = \so_{2n+1}$.

For convenience of reference later in this paper we state the following
theorem from \cite{BV1}.

\begin{Theorem} \label{T:BV}
  Let $\lambda \in \t^*_\Z$.  Then
 the associated variety to $\Ann_{U(\g)} L(\lambda)$ is equal to the nilpotent
 $G$-orbit with Jordan type given by $\BV(\lambda)$.
\end{Theorem}

\subsection{The $\tau$-equivalence} \label{ss:tau}
The Barbasch--Vogan Algorithm is used to find the associated variety of $\Ann_{U(\g)}(L(\lambda))$,
however in order to determine the action of the component group we need to be able
to determine when
$\Ann_{U(\g)} L(\mu) = \Ann_{U(\g)} L(\lambda)$.
This can be done using the $\tau$-equivalence.
This is an equivalence
relation on
the set of integral
weights of $\t$.

Recall our realization of $\g$ and its Borel subalgebra $\b$ defined in \S\ref{SS:realize},
and recall that $\Phi^+$ is the system of positive roots for $\g$ defined from $\b$. Let
$\Delta$ be the base of $\Phi$ corresponding to $\Phi^+$.
Also for $\alpha \in \Phi$, let $s_\alpha \in W$ denote the corresponding reflection
in the Weyl group $W$ of $\g$ with respect to $\t$.
For $w \in W$, let
\[
   S(w) = \{ \alpha  \in \Phi^+ \mid w \alpha \not \in \Phi^+ \}.
\]
Now let
\[
   \tau(w) = S(w) \cap \Delta.
\]

Suppose that $\lambda \in \t^*$ is an
integral antidominant weight.
Let $\alpha \in \Delta$ and $w \in W$.  Suppose that $\alpha \in \tau(w^{-1})$ satisfies
$\tau(w^{-1} s_\alpha) \not \subseteq \tau(w^{-1})$.
Then
\[
   \Ann_{U(\g)} L(s_\alpha w \lambda) = \Ann_{U(\g)} L(w \lambda).
\]
by \cite[Theorem 5.1]{Jo3}, see also \cite[Proposition 15]{BV1}.
With this in mind, we define the $\tau$-equivalence on integral weights to be the equivalence relation
generated by declaring that
$$
\lambda_1 \sim^\tau \lambda_2
$$
if for some antidominant integral weight  $\lambda'$, $w \in W$, and $\alpha \in \Delta$ the following hold:
$\lambda_1 = w \lambda'$, $\lambda_2 = s_\alpha w \lambda'$, and
$\tau(w^{-1} s_\alpha) \not \subseteq \tau(w^{-1})$.
In fact the below theorem, which is \cite[Theorem 3.5.9]{Ga3}, states that the
$\tau$-equivalence is a complete invariant on primitive ideals.

\begin{Theorem} \label{T:tau}
   Let $\lambda, \mu \in \t^*$ be integral weights.
  Then $\lambda \sim^\tau \mu$ if and only if
  $\Ann_{U(\g)} L(\lambda) = \Ann_{U(\g)} L(\mu)$.
\end{Theorem}

We identify the weight $\sum_{i=1}^n a_i \eps_i \in \t^*$ with the list $(a_1, \dots, a_n)$.
Then one can check that the $\tau$-equivalence is generated by
the following three relations:
\begin{itemize}
 \item[(R1)]
$(a_1, \dots, a_n)
\sim^\tau
(b_1, \dots, b_n)$
if $(a_1, \dots, a_n) \sim^K (b_1, \dots, b_n)$
\item[(R2)]
$(a_1, \dots, a_n) \sim^\tau (a_1, \dots, a_{n-1}, -a_n)$
if $|a_{n-1}| < |a_n|$.

\item[(R3)]
$(a_1, \dots, a_n) \sim^\tau
(a_1, \dots, a_{n-2}, a_n,a_{n-1})$ if
$a_{n-1} a_n < 0$.
\end{itemize}
In (R1) $\sim^K$
denotes Knuth equivalence, as defined in \cite[\S2]{Fu}.

The references for the results stated often only deal with the
case of regular weights. However, \cite[Lemma 5.6]{Ja} implies that they are
valid for non-regular weights too.

\section{The 3 row case for $\g = \sp_{2n}$} \label{S:Ccase}

Let $\g = \sp_{2n}$ and suppose that $\bp$ has three parts.
Then we write $\bp = (l^2,m)$, where $l$ must be even if $l < m$.
In this section we classify the finite dimensional $U(\g,e)$-modules, and we
use the $\tau$-equivalence to
describe the component group action on these modules.

Let $C$ be the component group of $e$, so
\[
C = \begin{cases}
     \lan c \ran \cong \Z_2 & \text{if $l$ is even and $l \ne m$;} \\
     1 & \text{otherwise}.
    \end{cases}
\]

The lemma below deals with the (easy) cases where $l$ is even and $l \leq m$, or
$l$ is odd (in which case $l > m$).

\begin{Lemma} \label{L:recsC}
    Suppose that
$A \in \sTab^\leq(P)$ and
$l$ is even and $l \leq m$,
or $l$ is odd.
   Then $L(A)$ is finite dimensional if and only if $A$ is justified row equivalent to column strict.
   Furthermore in the case that $l$ is even and $l < m$, if $L(A)$ is finite dimensional, then
   $c \cdot L(A) \cong L(A)$.
\end{Lemma}

\begin{proof}
First we consider the case that
$l$ is even and $l \leq m$.
So
$\content(l,l,m) = (\frac{l}{2}, \frac{l}{2}, \frac{m}{2}+1)$.
It is easy to see that the only partition with this content is $(l,l,m)$.
Therefore, by Theorem \ref{T:BGKt2} and Theorem \ref{T:recs} we have, $L(A)$ is finite dimensional if and only $\part(\RS(A)) = (l,l,m)$
if and only if
$A$ is justified row equivalent to column strict.
Now the statement about the action of $C$ follows from \ref{T:recs2}.

The case $l$ is odd is similar.
\end{proof}

So we are left to consider the case where
$l > m$ and $l$ is even.  Below we explain the action of $c$ on the
s-tables corresponding to finite dimensional $U(\g,e)$-modules.
We need to use the definition of the $\sharp$-special element of a list of integers,
which is given in \cite[\S6]{Bro2}.

Let $B \in \sTab^\leq(P')$ be an s-table for some s-frame $P'$ with an even number of rows.
If the $\sharp$-element of the upper middle row of $B$ is defined, then we
let $c'B$ denote the s-table $B' \in \sTab^\leq(P')$ where all the rows of $B'$
are the same as $B$,
except that in the upper middle row the $\sharp$-element is replaced by
its negative and the corresponding change to the lower middle row is also made;
otherwise we say the $c'B$ is undefined.

Let $a_1, \dots, a_l$ be the entries in the top row,
and let $b_1, \dots, b_{m/2}$ be the entries in the
first half of the middle row of $A$.
Let $A'$ be the s-table
with 4 rows of row lengths $l, m/2, m/2, l$,
where the top row has entries $a_1, \dots, a_l$ and the row below the top row has entries
$b_1, \dots, b_{m/2}$.

The rows of $A'$ are labelled by $1, 2, -2, -1$ from top to bottom.
We have the row swapping operators $\bar s_i$ from \S\ref{ss:rowswap} acting on $A'$;
for convenience in this section we do not include the $\star$ in the notation.
Let $B =  \bar s_1  c'   \bar s_1  A'$, provided that it is defined, otherwise $c \cdot A$
is undefined.

Let $d_1, \dots, d_l$ be the entries in the top row of $B$, and let $e_1, \dots, e_m$ be the entries in the row below the top row of $B$.
If
$e_1, \dots, e_m$ are not all negative, then we say that $c \cdot A$ is undefined.
Otherwise
we declare that $c \cdot A$ is the s-table with row lengths $(l, m, l)$ where the top row has entries
$d_1, \dots, d_l$ and the middle row has entries $e_1, \dots, e_m, -e_m, \dots, -e_1$.

For example,
if
\[
A =
\begin{array}{c}
\begin{picture}(80,60)
   \put(0,0){\line(1,0){80}}
   \put(0,20){\line(1,0){80}}
   \put(0,40){\line(1,0){80}}
   \put(0,60){\line(1,0){80}}
   \put(0,0){\line(0,1){20}}
   \put(0,40){\line(0,1){20}}
   \put(20,0){\line(0,1){60}}
   \put(40,0){\line(0,1){60}}
   \put(60,0){\line(0,1){60}}
   \put(80,0){\line(0,1){20}}
   \put(80,40){\line(0,1){20}}
 \put(8,10){\makebox(0,0){{-5}}}
 \put(28,10){\makebox(0,0){{-4}}}
 \put(48,10){\makebox(0,0){{-3}}}
 \put(68,10){\makebox(0,0){{-2}}}
 \put(28,30){\makebox(0,0){{-1}}}
 \put(50,30){\makebox(0,0){{1}}}
 \put(10,50){\makebox(0,0){{2}}}
 \put(30,50){\makebox(0,0){{3}}}
 \put(50,50){\makebox(0,0){{4}}}
 \put(70,50){\makebox(0,0){{5}}}
\end{picture}
\end{array}
\quad
\text{then}
\quad
A' =
\begin{array}{c}
\begin{picture}(80,80)
   \put(0,0){\line(1,0){80}}
   \put(0,20){\line(1,0){80}}
   \put(0,40){\line(1,0){20}}
   \put(0,60){\line(1,0){80}}
   \put(0,80){\line(1,0){80}}
   \put(0,0){\line(0,1){80}}
   \put(20,0){\line(0,1){80}}
   \put(40,0){\line(0,1){20}}
   \put(60,0){\line(0,1){20}}
   \put(80,0){\line(0,1){20}}
   \put(40,60){\line(0,1){20}}
   \put(60,60){\line(0,1){20}}
   \put(80,60){\line(0,1){20}}

 \put(8,10){\makebox(0,0){{-5}}}
 \put(28,10){\makebox(0,0){{-4}}}
 \put(48,10){\makebox(0,0){{-3}}}
 \put(68,10){\makebox(0,0){{-2}}}
 \put(10,30){\makebox(0,0){{1}}}
 \put(8,50){\makebox(0,0){{-1}}}
 \put(10,70){\makebox(0,0){{2}}}
 \put(30,70){\makebox(0,0){{3}}}
 \put(50,70){\makebox(0,0){{4}}}
 \put(70,70){\makebox(0,0){{5}}}
\end{picture}
\end{array}.
\]
So
\[
 \bar s_1  A' =
\begin{array}{c}
\begin{picture}(80,80)
   \put(0,0){\line(1,0){20}}
   \put(0,20){\line(1,0){80}}
   \put(0,40){\line(1,0){80}}
   \put(0,60){\line(1,0){80}}
   \put(0,80){\line(1,0){20}}
   \put(0,0){\line(0,1){80}}
   \put(20,0){\line(0,1){80}}
   \put(40,20){\line(0,1){20}}
   \put(60,20){\line(0,1){20}}
   \put(80,20){\line(0,1){20}}
   \put(40,40){\line(0,1){20}}
   \put(60,40){\line(0,1){20}}
   \put(80,40){\line(0,1){20}}

 \put(8,10){\makebox(0,0){{-2}}}
 \put(8,30){\makebox(0,0){{-5}}}
 \put(28,30){\makebox(0,0){{-4}}}
 \put(48,30){\makebox(0,0){{-3}}}
 \put(70,30){\makebox(0,0){{1}}}
 \put(8,50){\makebox(0,0){{-1}}}
 \put(10,70){\makebox(0,0){{2}}}
 \put(30,50){\makebox(0,0){{3}}}
 \put(50,50){\makebox(0,0){{4}}}
 \put(70,50){\makebox(0,0){{5}}}
\end{picture}
\end{array}
\quad
\text{and}
\quad
c'   \bar s_1  A' =
\begin{array}{c}
\begin{picture}(80,80)
   \put(0,0){\line(1,0){20}}
   \put(0,20){\line(1,0){80}}
   \put(0,40){\line(1,0){80}}
   \put(0,60){\line(1,0){80}}
   \put(0,80){\line(1,0){20}}
   \put(0,0){\line(0,1){80}}
   \put(20,0){\line(0,1){80}}
   \put(40,20){\line(0,1){20}}
   \put(60,20){\line(0,1){20}}
   \put(80,20){\line(0,1){20}}
   \put(40,40){\line(0,1){20}}
   \put(60,40){\line(0,1){20}}
   \put(80,40){\line(0,1){20}}
 \put(8,10){\makebox(0,0){{-2}}}
 \put(8,30){\makebox(0,0){{-4}}}
 \put(28,30){\makebox(0,0){{-3}}}
 \put(50,30){\makebox(0,0){{1}}}
 \put(70,30){\makebox(0,0){{5}}}
 \put(8,50){\makebox(0,0){{-5}}}
 \put(10,70){\makebox(0,0){{2}}}
 \put(28,50){\makebox(0,0){{-1}}}
 \put(50,50){\makebox(0,0){{3}}}
 \put(70,50){\makebox(0,0){{4}}}
\end{picture}
\end{array}.
\]
Hence
\[
 \bar s_1  c'  \bar s_1  A' =
\begin{array}{c}
\begin{picture}(80,80)
   \put(0,0){\line(1,0){80}}
   \put(0,20){\line(1,0){80}}
   \put(0,40){\line(1,0){20}}
   \put(0,60){\line(1,0){80}}
   \put(0,80){\line(1,0){80}}
   \put(0,0){\line(0,1){80}}
   \put(20,0){\line(0,1){80}}
   \put(40,0){\line(0,1){20}}
   \put(60,0){\line(0,1){20}}
   \put(80,0){\line(0,1){20}}
   \put(40,60){\line(0,1){20}}
   \put(60,60){\line(0,1){20}}
   \put(80,60){\line(0,1){20}}

 \put(8,10){\makebox(0,0){{-4}}}
 \put(28,10){\makebox(0,0){{-3}}}
 \put(48,10){\makebox(0,0){{-2}}}
 \put(70,10){\makebox(0,0){{5}}}
 \put(10,30){\makebox(0,0){{1}}}
 \put(8,50){\makebox(0,0){{-1}}}
 \put(8,70){\makebox(0,0){{-5}}}
 \put(30,70){\makebox(0,0){{2}}}
 \put(50,70){\makebox(0,0){{3}}}
 \put(70,70){\makebox(0,0){{4}}}
\end{picture}
\end{array},
\quad
\text{so}
\quad
c \cdot A =
\begin{array}{c}
\begin{picture}(80,60)
   \put(0,0){\line(1,0){80}}
   \put(0,20){\line(1,0){80}}
   \put(0,40){\line(1,0){80}}
   \put(0,60){\line(1,0){80}}
   \put(0,0){\line(0,1){20}}
   \put(0,40){\line(0,1){20}}
   \put(20,0){\line(0,1){60}}
   \put(40,0){\line(0,1){60}}
   \put(60,0){\line(0,1){60}}
   \put(80,0){\line(0,1){20}}
   \put(80,40){\line(0,1){20}}
 \put(8,10){\makebox(0,0){{-4}}}
 \put(28,10){\makebox(0,0){{-3}}}
 \put(48,10){\makebox(0,0){{-2}}}
 \put(70,10){\makebox(0,0){{5}}}
 \put(28,30){\makebox(0,0){{-1}}}
 \put(50,30){\makebox(0,0){{1}}}
 \put(8,50){\makebox(0,0){{-5}}}
 \put(30,50){\makebox(0,0){{2}}}
 \put(50,50){\makebox(0,0){{3}}}
 \put(70,50){\makebox(0,0){{4}}}
\end{picture}
\end{array}.
\]

The following lemma follows from \cite[Remark 5.8]{BG1}.

\begin{Lemma} \label{L:cctau}
 Let $A \in \sTab^\leq (P)$ and suppose $c \cdot A$ is defined.
 Then $\word(A) \sim^\tau \word(c \cdot A)$.
\end{Lemma}

Our next goal is to prove that $c \cdot A$ is defined when $A$ corresponds to a finite dimensional
$U(\g,e)$-module.

\begin{Lemma} \label{L:cdefined}
   Let $A \in \sTab^\leq (P)$.
	If $L(A)$ is finite dimensional
	then $c \cdot A$ is defined.
\end{Lemma}

\begin{proof}
	Let $a_1, \dots, a_l$ be the top row of $A$, and
	let $b_1, \dots, b_{m/2}$ be the first half of
	the middle row of $A$.
	Since $L(A)$ is finite dimensional we must have that
	$\content(\part(\RS(A))) = \content(l,l,m) = (\frac{m}{2},\frac{l}{2},\frac{l+2}{2})$.
This gives that $\part(\RS(A))$ must be
	$(l,l,m), (l+1, l-1, m),$ or $(l,l-1,m+1)$.
The last of these we can rule out by
Lemma \ref{L:altRS}.
Thus
$\part(\RS(A))$ =
	$(l,l,m)$ or $(l+1, l-1, m)$.
	In either case we note that $ \bar s_1  A'$ is defined,
otherwise we would have that
	for some $i \ge 0$ that $a_{l/2 -i} < b_{m/2-i}$, in which case we have the increasing subword
	$a_1, \dots, a_{l/2-i}, b_{m/2-i}, \dots, b_{m/2},
	-b_{m/2}, \dots, -b_{m/2-i}, -a_{l/2-i}, \dots, -a_1$
	 of length
	$l+2$, which contradicts Lemma \ref{L:altRS}.

	Now suppose that $\part(\RS(A)) = (l,l,m)$.
	Then by Theorem \ref{T:recs} $A$ is row equivalent to column strict,
	so we have that $a_{i} + a_{m-i+1} > 0$ for all $i$.
	Let $a_1' \leq \dots \leq a_{m}'$ be the elements from the top row which
	best fit over $b_1, \dots, b_{m/2}, -b_{m/2}, \dots, -b_1$.
	Let $a_1'', \dots, a_{l-m}''$ be the remaining elements of the top row.
	Then for $i=1,\dots, m/2$ we have that
	$- a_{m-i+1}' < b_i < a_i'$.
	Since $A$ is row equivalent to column strict,
we also have that $a_{i}'' + a_{l-m+1-i}'' > 0$ for all $i$.
	This shows that the $\sharp$-element of
	$a_{1}'', \dots, a_{l-m}'', b_1, \dots, b_{m/2}, a_{m/2+1}', \dots, a_{m}'$
is defined, and is greater than or
	equal to $0$.
	It also implies that the elements of
	$(a_{1}'', \dots, a_{l-m}'', b_1, \dots, b_{m/2}, a_{m/2+1}', \dots, a_{m}')^\sharp$
which best fit under $a_1', \dots, a_{m/2}'$
	are all negative.  Thus the $c \cdot A$ is defined.

	Now suppose that $\part(\RS(A)) = (l+1, l-1, m)$.
	By \cite[Lemma 5.6]{BG1} the $\sharp$-element of row 2 of $ \bar s_1  A'$ is defined, otherwise
	we could find an increasing subword of length $l+2$.
	Also the $\sharp$-element must be negative, otherwise
	we could not find an increasing subword of length
	$l+1$ in $\word( \bar s_1  A')$, since the middle two rows of $ \bar s_1  A'$ would then be column strict.

Next we need to prove that the action of $ \bar s_1$  is defined on $c' \bar s_1 A'$.
If it was not, then we could find
	two disjoint increasing strings of length $l+1$ in $\word(c' \bar s_1 A')$, which is a contradiction since
$\word(c'  \bar s_1 A')$ is $\tau$-equivalent to $\word(A)$; confer Theorem \ref{T:tau}.

	Finally we need to argue why the elements of row $2$ of
$c'  \bar s_1  A'$,
which best fit under row $1$
	are all negative.
If one the best fitting elements, say $b$ was positive, then we could form the following decreasing
chain: $a,b,-b,-a$ where $a$ is any element of row $1$ of $A'$ which is larger than $b$.
This contradicts the fact that $\part(\RS( \bar s_1  c' \bar s_1  A')) = (l,l,m)$ or $(l+1,l-1,m)$.
\end{proof}

We are now ready for the main theorem of this section.

\begin{Theorem} \label{T:caction1}
Suppose that $l$ is even and $l > m$, and let
$A \in \sTab^\leq(P)$.
   Then $L(A)$ is finite dimensional if and only if
$A$ is $C$-conjugate to an s-table that is justified row equivalent to column strict.
   Furthermore, if $L(A)$ is finite dimensional, then
   $c \cdot L(A) \cong L(c \cdot A)$.
\end{Theorem}

\begin{proof}
From the proof of the previous lemma, we know that if $L(A)$ is finite dimensional,
then $\part(\RS(B))$ is  $(l,l,m)$ or $(l+1,l-1,m)$.  In the former we case we have
that $A$ is row equivalent to column strict by Theorem \ref{T:recs}. In the latter
case we can see that $c \cdot A$ is row equivalent to column strict immediately from the
previous lemma and the following observation:  Suppose
$B \in \sTab^\leq(P)$ is such that $\part(\RS(B)) = (l,l,m)$ or $(l+1,l-1,m)$
and the middle two rows of $B'$ are row equivalent to column strict.  Then
$\part(\RS(B)) = (l,l,m)$.
Indeed, if we left justify the top two rows of $B'$ and right justify the bottom two rows then the resulting diagram is column strict, so it is impossible
to find an increasing chain of length $l+1$.

Now we prove the statement about the action of $c$.
Suppose that $L(A)$ is finite dimensional and assume that
$\part(\RS(A)) = (l,l,m)$.
We have that $c \cdot L(A) \cong L(B)$ for some $B$.  If $\part(\RS(B)) = (l,l,m)$ then $A=B$ by
Theorem \ref{T:recs}, and in this case it follows that we have $c \cdot A = B$.
If $\part(\RS(B)) = (l+1,l-1,m)$ then $\part(\RS(c \cdot B)) = (l,l,m)$, and since
$L(c \cdot B)$ and $L(A)$ are associated to the same primitive ideal of $U(\g)$ by Lemma \ref{L:cctau}
and $c \cdot B$ and $A$ are both row equivalent to column strict, then it follows from
\S\ref{SS:Losev} and Theorem \ref{T:recs2} that we must have that
$A = c \cdot B$.
\end{proof}

Last in this section we give the following lemma, which we need in the proof of Theorem \ref{T:mainintro}.

\begin{Lemma} \label{L:Clargersmaller}
If $A \in \sTab^\leq(P)$ is row equivalent to column strict, then $\word(c \cdot A)$
   can be obtained from $\word(A)$ through a series of
Knuth equivalences and larger-smaller transpositions.
In particular, $\part(\RS(A)) \leq \part(\RS(c \cdot A))$.
\end{Lemma}

\begin{proof}
   This is proven in \cite[Remark 5.8]{BG1}.
\end{proof}

\section{The 3 row case for $\g = \so_{2n+1}$} \label{S:Bcase}
Let $\g = \so_{2n}$ and suppose that $\bp$ has three parts.
Then we write $\bp = (l^2,m)$, where $l$ must be odd if $l > m$.
In this section we classify the finite dimensional $U(\g,e)$-modules, and we
use the $\tau$-equivalence to
describe the component group action on these modules.

Let $C$ be the component group of $e$, so
\[
C = \begin{cases}
     \lan c \ran \cong \Z_2 & \text{if $l$ is odd and $l \ne m$;} \\
     1 & \text{otherwise}.
    \end{cases}
\]

The lemma below deals with the (easy) cases, where $l > m$ (in which case $l$ must be odd), and $l \le m$ and is even. The proof is very similar to that of Lemma \ref{L:recsC} so is omitted.

\begin{Lemma} \label{L:recsB}
Suppose that $l$ is even,
or $l$ is odd and $l \ge m$.
Let $A \in \sTab^\leq(P)$.
Then
$L(A)$ is finite dimensional if and only if $A$ is justified row equivalent to column strict.
 Furthermore in the case that $l$ is odd and $l > m$, if $L(A)$ is finite dimensional, then
   $c \cdot L(A) \cong L(A)$.
\end{Lemma}

So we are left to consider the case where $l$ is odd and $m >l$, in this case we let $l = 2p+1$ and $m = 2q+1$, where $q >p$.
In the next few paragraphs we set up the combinatorics to describe the action
of $c$ on elements of $\sTab^\le(P)$ corresponding to finite dimensional
representations.

Let $A \in \sTab^\le(P)$.
Let $a_1, \dots, a_{2p +1}$ be the top row  of $A$, and let
$b_1, \dots, b_q,0,-b_q, \dots, -b_1$ be the middle row.
We define two tables from $A$, $A^{L^+}$ and $A^{L^-}$, in the following manner.
$A^{L^+}$ is the left justified 3 row table with row $1$ equal to
$a_1, \dots, a_{p+1}$, with row $2$ equal to
$b_1, \dots, b_q$, and with row $3$ equal to
$-a_{2p+1}, \dots, -a_{p+2}$.
$A^{L^-}$ is the left justified 3 row table with row $1$ equal to
$a_1, \dots, a_{p}$, with row $2$ equal to
$b_1, \dots, b_q$, and with row $3$ equal to
$-a_{2p+1}, \dots, -a_{p+1}$.

We define the $C$-action on $A$ in the following manner depending on cases below.  Here we use the
row swapping operations for tables mentioned in \S\ref{ss:rowswap}, and we omit the $\star$ in the notation
for convenience.

\noindent Case 1: If $A^{L^-}$ is row equivalent to column strict then we define
$c \cdot A = B$, where $B$ is the unique stable in $\sTab^{\le}(P)$
such that
$B^{L^+} = s_2 s_1 s_2 A^{L^-}$.

\noindent Case 2: If $A^{L^-}$ is not row equivalent to column strict, but $A^{L^+}$ is row equivalent to column strict then we define
$c \cdot A = B$, where $B$ is the unique stable in $\sTab^{\le}(P)$
such that
$B^{L^-} = s_2 s_1 s_2 A^{L^+}$, provided that such an s-table exists; note that $B$ exists precisely when $s_1 s_2 A^{L^+}$ contains only negative numbers
in row $2$, and this will {\em not} happen if
$A^{L^-}$ is row equivalent to column strict.  If such $B$ does not exist, then we say that
$c \cdot A$ is not defined.

\noindent Case 3: If neither $A^{L^-}$ nor $A^{L^+}$ is row equivalent to column strict, then we say that $c \cdot A$
is undefined.

For example,
   suppose that
\[
  A =
\begin{array}{c}
   \begin{picture}(100,60)
    \put(20,0){\line(1,0){60}}
    \put(0,20){\line(1,0){100}}
    \put(0,40){\line(1,0){100}}
    \put(20,60){\line(1,0){60}}
    \put(0,20){\line(0,1){20}}
    \put(20,0){\line(0,1){60}}
    \put(40,0){\line(0,1){60}}
    \put(60,0){\line(0,1){60}}
    \put(80,0){\line(0,1){60}}
    \put(100,20){\line(0,1){20}}

\put(28,10){\makebox(0,0){{-6}}}
\put(48,10){\makebox(0,0){{-5}}}
\put(70,10){\makebox(0,0){{2}}}
\put(8,30){\makebox(0,0){{-3}}}
\put(28,30){\makebox(0,0){{-1}}}
\put(50,30){\makebox(0,0){{0}}}
\put(70,30){\makebox(0,0){{1}}}
\put(90,30){\makebox(0,0){{3}}}
\put(28,50){\makebox(0,0){{-2}}}
\put(50,50){\makebox(0,0){{5}}}
\put(70,50){\makebox(0,0){{6}}}
   \end{picture}
\end{array}.
\]
Then
\[
 A^{L^+} =
\begin{array}{c}
\begin{picture}(40,60)
 \put(0,0){\line(1,0){20}}
 \put(0,20){\line(1,0){40}}
 \put(0,40){\line(1,0){40}}
 \put(0,60){\line(1,0){40}}
 \put(0,0){\line(0,1){60}}
 \put(20,0){\line(0,1){60}}
 \put(40,20){\line(0,1){40}}
\put(8,10){\makebox(0,0){{-6}}}
\put(8,30){\makebox(0,0){{-3}}}
\put(28,30){\makebox(0,0){{-1}}}
\put(8,50){\makebox(0,0){{-2}}}
\put(30,50){\makebox(0,0){{5}}}
\end{picture}
\end{array}
\qquad \text{and} \qquad
 A^{L^-} =
\begin{array}{c}
\begin{picture}(40,60)
 \put(0,0){\line(1,0){40}}
 \put(0,20){\line(1,0){40}}
 \put(0,40){\line(1,0){40}}
 \put(0,60){\line(1,0){20}}
 \put(0,0){\line(0,1){60}}
 \put(20,0){\line(0,1){60}}
 \put(40,0){\line(0,1){40}}
\put(8,10){\makebox(0,0){{-6}}}
\put(28,10){\makebox(0,0){{-5}}}
\put(8,30){\makebox(0,0){{-3}}}
\put(28,30){\makebox(0,0){{-1}}}
\put(8,50){\makebox(0,0){{-2}}}
\end{picture}
\end{array}.
\]

Since $A^{L^-}$ is column strict,
we are in Case 1.  Now
\[
 s_2 s_1 s_2 A^{L^-} =
\begin{array}{c}
\begin{picture}(40,60)
 \put(0,0){\line(1,0){20}}
 \put(0,20){\line(1,0){40}}
 \put(0,40){\line(1,0){40}}
 \put(0,60){\line(1,0){40}}
 \put(0,0){\line(0,1){60}}
 \put(20,0){\line(0,1){60}}
 \put(40,20){\line(0,1){40}}
\put(8,10){\makebox(0,0){{-5}}}
\put(8,30){\makebox(0,0){{-6}}}
\put(28,30){\makebox(0,0){{-3}}}
\put(8,50){\makebox(0,0){{-2}}}
\put(28,50){\makebox(0,0){{-1}}}
\end{picture}
\end{array},
\]
so
\[
  c \cdot A =
\begin{array}{c}
   \begin{picture}(100,60)
    \put(20,0){\line(1,0){60}}
    \put(0,20){\line(1,0){100}}
    \put(0,40){\line(1,0){100}}
    \put(20,60){\line(1,0){60}}
    \put(0,20){\line(0,1){20}}
    \put(20,0){\line(0,1){60}}
    \put(40,0){\line(0,1){60}}
    \put(60,0){\line(0,1){60}}
    \put(80,0){\line(0,1){60}}
    \put(100,20){\line(0,1){20}}

\put(28,10){\makebox(0,0){{-5}}}
\put(50,10){\makebox(0,0){{1}}}
\put(70,10){\makebox(0,0){{2}}}
\put(8,30){\makebox(0,0){{-6}}}
\put(28,30){\makebox(0,0){{-3}}}
\put(50,30){\makebox(0,0){{0}}}
\put(70,30){\makebox(0,0){{3}}}
\put(90,30){\makebox(0,0){{6}}}
\put(28,50){\makebox(0,0){{-2}}}
\put(48,50){\makebox(0,0){{-1}}}
\put(70,50){\makebox(0,0){{5}}}
   \end{picture}
\end{array}.
\]

We need to prove that $\word(A)$ is $\tau$-equivalent to $\word(c \cdot A)$.  To do this we need the following lemmas.

\begin{Lemma} \label{L:flat1}
	Let $a, b_1, \dots, b_m$ be such that $a > 0$, $b_1 < \dots < b_m < 0$ and $-a < b_m$.
	Then
	\[
	(a,b_1, \dots, b_m) \sim^\tau
	(b_1, , \dots, b_m, -a).
	\]
\end{Lemma}
\begin{proof}
	By applying the Robinson--Schensted Algorithm we see that
	$(a,b_1, \dots, b_m)$ is Knuth equivalent to
	$(b_1, \dots, b_{m-1}, a, b_m)$.
	By applying the relations (R3) then (R2) from the definition of the $\tau$-equivalence
	we get that this is $\tau$-equivalent to
	$(b_1, \dots, b_m,-a)$.
\end{proof}

For positive integers $k,m$ we define an operation $\LT_{k,m}$ on certain lists.
Suppose that  $(a_1, \dots, a_l, b_1, \dots, b_m)$
is a list such that
$l \ge 2k-1$,
$m \ge k$, $b_m < 0$, $a_{l-k} > 0$,
and
the table
\begin{equation} \label{EQ:LT}
B =
\begin{array}{|c|c|c|c|c|c|c|}
	\cline {1-4}
	a_{l-2k+2} & a_{l-2k+1} & \dots & a_{l-k} \\
	\hline
	b_1 & b_2 & \dots & b_{k-1} & b_k & \dots & b_m \\
	\hline
	-a_l & -a_{l-1} & \dots & -a_{l-k+2} & -a_{l-k+1} \\
	\cline{1-5}
\end{array}
\end{equation}
is row equivalent to column strict with increasing rows.
We define
	$\LT_{k,m} (a_1, \dots, a_{l}, b_1, \dots, b_m)$
	to be the list
	$(a_1, \dots, a_{l-2k+1})$
	concatenated with $\word(B)$.
For example,
if $A \in \sTab^\leq(P)$ is justified row equivalent to column strict, $P$ has row lengths $(2p+1,2q+1,2p+1)$,
and
\[
\word(A) = (a_1, \dots, a_{2 p+1}, b_1, \dots, b_{2q+1},0,-b_{2q+1}, \dots, -b_1, -a_{2p+1}, \dots, -a_1),
\]
then
$\LT_{p+1,2q+1}(a_1, \dots, a_{2p+1}, b_1, \dots, b_{2q+1})= \word(A^{L^-})$.

We would also like to explicitly describe $\LT_{k,m}^{-1}$.
This will be defined on lists of the form
$(a_1, \dots, a_l, b_1, \dots, b_m, c_1, \dots, c_k)$
where
$m \ge k$, $l \ge k-1$,
$c_k < 0$, $-c_k > a_l$, $b_m < 0$,
and the following table is row equivalent to column strict with increasing rows:
\begin{equation} \label{EQ:LT2}
B =
\begin{array}{|c|c|c|c|c|c|c|}
	\cline {1-4}
	a_{l-k+2} & a_{l-k+1} & \dots & a_{l} \\
        \hline
	b_1 & b_2 & \dots & b_{k-1} & b_k & \dots & b_m \\
	\hline
	c_1 & c_2 & \dots & c_{k-1} & c_k  \\
	\cline {1-5}
\end{array}.
\end{equation}
Now
$\LT_{k,m}^{-1}(a_1, \dots, a_l, c_1, \dots, c_k, b_1, \dots, b_m)
=(a_1, \dots, a_{l}, -c_k, \dots, -c_1, b_1, \dots, b_m)$.
If any of the above conditions are not met then we say that
$\LT_{k,m}^{-1}(a_1, \dots, a_l, b_1, \dots, b_m, c_1, \dots, c_k)$ is undefined.

\begin{Lemma} \label{L:flat2}
Let
$(a_1, \dots, a_l, b_1, \dots, b_m)$
be a list on which $\LT_{k,m}$ is defined.
Then
\[
(a_1, \dots, a_l, b_1, \dots, b_m) \sim^\tau
\LT_{k,m}(a_1, \dots, a_l, b_1, \dots, b_m).
\]
\end{Lemma}
\begin{proof}
We may assume that $l = 2k-1$.
	We proceed by induction on $k$,
The case $k=1$ is given by
	Lemma \ref{L:flat1}.
Now since
\begin{equation} \label{EQ:hLT}
\begin{array}{|c|c|c|c|c|c|c|}
	\cline {1-4}
	a_{1} & a_{2} & \dots & a_{k-1} \\
	\hline
	b_1 & b_2 & \dots & b_{k-1} & b_k & \dots & b_m \\
	\hline
	-a_{2k-1} & -a_{2k-2} & \dots & -a_{k+1} & -a_{k} \\
	\cline{1-5}
\end{array}
\end{equation}

is row equivalent to column strict, we also have that
\[
\begin{array}{|c|c|c|c|c|c|c|}
	\cline {1-4}
	a_{3} & a_{4} & \dots & a_{k} \\
	\hline
	b_1 & b_2 & \dots & b_{k-2} & b_{k-1} & \dots & b_m \\
	\hline
	-a_{2k-1} & -a_{2k-2} & \dots & -a_{k+2} & -a_{k+1} \\
	\cline{1-5}
\end{array}
\]
is row equivalent to column strict.
 So by induction,
        $(a_1, \dots, a_l, b_1, \dots, b_m)$ is $\tau$-equivalent to
        $\LT_{k-1,m}(a_1, \dots, a_l) =
(a_1, \dots, a_k, b_1, \dots, b_m, -a_{2k-1}, \dots, -a_{k+1})$.
Now let $b_{i_1}, \dots, b_{i_{k-1}}$ be the elements of $b_1, \dots, b_m$ which best fit
over $-a_{2k-1}, \dots, -a_{k+1}$.
Thus
\begin{equation} \label{EQ:hLT2}
  (a_1, \dots, a_k, b_1, \dots, b_m, -a_{2k-1}, \dots, -a_{k+1}) \sim^K
(a_1, \dots, a_k, b_{i_1}, \dots, b_{i_{k-1}}, a'_1, \dots, a'_m),
\end{equation}
where $(a'_1, \dots, a'_m)$ is the sorted list consisting of
$-a_{2k-1}, \dots, -a_{k+1}$ and $\{b_l \mid l \neq i_j$ for $j=1,\dots, k-1\}$, and
$\sim^K$ denotes Knuth equivalence.
Now from \eqref{EQ:hLT} we can see that
$b_{i_1}, \dots, b_{i_{k-1}}$ best fits under $a_1, \dots, a_{k-1}$, so
\[
(a_1, \dots, a_k, b_{i_1}, \dots, b_{i_{k-1}}, a'_1, \dots, a'_m) \sim^K
(a_1, \dots, a_{k-1}, b_{i_1}, \dots, b_{i_{k-1}}, a_k, a'_1, \dots, a'_m).
\]
We also get from \eqref{EQ:hLT} that $a'_m = -b_m$,
so by Lemma \ref{L:flat1} we have that
\[
(a_1, \dots, a_{k-1}, b_{i_1}, \dots, b_{i_{k-1}}, a_k, a'_1, \dots, a'_m) \sim^\tau
(a_1, \dots, a_{k-1}, b_{i_1}, \dots, b_{i_{k-1}}, a'_1, \dots, a'_m, -a_k).
\]
Finally we can use the Knuth equivalence in \eqref{EQ:hLT2} to get that this is
Knuth equivalent to
\[
(a_1, \dots, a_{k-1}, b_{1}, \dots, b_{m}, -a_{2k-1}, \dots, -a_k).
\]
\end{proof}

\begin{Lemma} \label{L:tech1}
 Suppose that we are given a skew symmetric word
\[
w = (a,b_1, \dots, b_m, c,0,-c,-b_m,\dots,-b_1,-a)
\]
such that $\part(\RS(a,b_1,\dots,b_m,c)) = (m,1,1)$, $b_1 < b_2 < \dots < b_m < 0$, $c <0$,
and $-c > a$.
Then
\[
w \sim^K
(a,b_1, \dots, b_m, -c,0,c,-b_m,\dots,-b_1,-a).
\]
\end{Lemma}
\begin{proof}
Calculate $\RS(a,b_1, \dots, b_m, c,0,-c,-b_m)$, and
$\RS(a,b_1, \dots, b_m,-c,0,c,-b_m)$, then observe that they are equal.
\end{proof}

\begin{Lemma} \label{L:LT}
Suppose that we are given a skew symmetric word
\[
w = (a_1, \dots, a_l,b_1, \dots, b_m, c_1, \dots,c_k,0,-c_k, \dots, -c_1,-b_m,\dots,-b_1,-a_l, \dots, -a_1)
\]
such that
$k \leq l \leq m$,
$a_1 < a_2 < \dots < a_l, b_1 < b_2 < \dots < b_m < 0, c_1 < c_2 < \dots < c_k < 0$, $-c_k > a_l$, and
$\part(\RS(a_1, \dots, a_l,b_1, \dots, b_m, c_1, \dots, c_k)) = (m,l,k)$.
Then
\[
 w \sim^K
(a_1, \dots, a_l,b_1, \dots, b_m, -c_k, \dots,-c_1,0,c_1, \dots, c_k,-b_m,\dots,-b_1,-a_l, \dots, -a_1)
\]
and
\[
(a_1, \dots, a_l,b_1, \dots, b_m, c_1, \dots,c_k) \sim^\tau
(a_1, \dots, a_l,b_1, \dots, b_m, -c_k, \dots,-c_1).
\]
\end{Lemma}
\begin{proof}
We prove this by induction on $k$.  The case $k=1$ is given by Lemma \ref{L:tech1} and
from condition (R2) in the definition of the $\tau$-equivalence.

To prove the general case, first we best fit
$c_1, \dots, c_{k-1}$ under $b_1, \dots, b_m$, which gives that
\begin{equation} \label{EQ:K1}
  (b_1, \dots, b_m, c_1, \dots, c_{k-1}) \sim^K
(b_{i_1}, \dots, b_{i_{k-1}}, c'_1, \dots, c'_m).
\end{equation}
Now we can best fit $b_{i_1}, \dots, b_{i_{k-1}}$ under $a_1, \dots, a_l$ to get that
\begin{equation} \label{EQ:K2}
(a_1, \dots, a_l,b_{i_1}, \dots, b_{i_{k-1}}) \sim^K
(a_{i'_1}, \dots, a_{i'_{k-1}}, b'_1, \dots, b'_l).
\end{equation}
Putting this all together, we get that $w$ is Knuth equivalent to
\[
(a_{i'_1}, \dots, a_{i'_{k-1}}, b'_1, \dots, b'_l,c'_1, \dots, c'_m, c_k,0,-c_k, -c'_m, \dots, -c'_1, -b'_l, \dots, -b'_1, -a_{i'_{k-1}}, \dots, -a_{i'_1}).
\]

Since $\part(\RS(a_1, \dots, a_l,b_1, \dots, b_m, c_1, \dots, c_k)) = (m,l,k)$,
we can deduce that $b'_l = a_l$ and $c'_m = b_m.$  We can also use this to deduce that the element
of $(b_1, \dots, b_m)$ which best fits over $c_k$ is an element of $(c'_1, \dots, c'_m)$.
Now we apply Lemma \ref{L:tech1} to the part of this word between $b'_l$ and $-b'_l$ to get that this is Knuth equivalent to
\[
(a_{i'_1}, \dots, a_{i'_{k-1}}, b'_1, \dots, b'_l,c'_1, \dots, c'_m, -c_k,0,c_k, -c'_m, \dots, -c'_1, -b'_l, \dots, -b'_1, -a_{i'_{k-1}}, \dots, -a_{i'_1}).
\]
Now we can apply the Knuth equivalences in \eqref{EQ:K1} and \eqref{EQ:K2} to get that this is Knuth equivalent to
\[
(a_1, \dots, a_l,b_1, \dots, b_m, c_1, \dots,c_{k-1}, -c_k,0,c_k, -c_{k-1},\dots, -c_1,-b_m,\dots,-b_1,-a_l, \dots, -a_1).
\]
Now we can best fit
$b_{m-k+2}, \dots, b_m$
over
$c_1, \dots, c_{k-1}, -c_k$ to get
\[
   (b_{m-k+2}, \dots, b_m, c_1, \dots, c_{k-1}, -c_k) \sim^K
   (b_{m-k+2}, \dots, b_m, -c_k, c_1, \dots, c_{k-1}).
\]
Now we can best first $a_1, \dots, a_l$ over $b_1, \dots, b_m,-c_k$ to get that
\begin{equation} \label{EQ:K3}
(a_1, \dots, a_l,b_1, \dots, b_m, -c_k) \sim^K (a'_1, \dots, a'_m,-c_k,b_{j_1}, \dots, b_{j_l}).
\end{equation}
So we have that
\[
(a_1, \dots, a_l,b_1, \dots, b_m, c_1, \dots,c_{k-1}, -c_k,0,c_k, -c_{k-1},\dots, -c_1,-b_m,\dots,-b_1,-a_l, \dots, -a_1),
\]
and therefore $w$, is Knuth equivalent to
\[
(a'_1, \dots, a'_m,-c_k,b_{j_1}, \dots, b_{j_l},c_1, \dots, c_{k-1},0,-c_{k-1}, \dots, -c_1, -b_{j_l}, \dots, -b_{j_1}, c_k, -a'_m, \dots, -a'_1).
\]
By induction this is Knuth equivalent to
\[
(a'_1, \dots, a'_m,-c_k,b_{j_1}, \dots, b_{j_l},-c_{k-1}, \dots, -c_1,0,c_1, \dots, c_{k-1}, -b_{j_l}, \dots, -b_{j_1}, c_k, -a'_m, \dots, -a'_1).
\]
Finally, by applying the Knuth equivalence \eqref{EQ:K3}, we get that this is Knuth equivalent to
\[
(a_1, \dots, a_l,,b_1, \dots, b_m,-c_{k}, \dots, -c_1,0,c_1, \dots, c_{k}, -b_m, \dots, -b_1, -a_l, \dots, -a_1).
\]
\end{proof}

\begin{Theorem} \label{T:LT3}
  Let $A \in \sTab^\leq(P)$ be justified row equivalent to column strict.
  Let
\[
(a_1, \dots, a_{q+1}, b_1, \dots, b_p, -a_{2q+1}, \dots, -a_{q+2}) = \word(s_2 s_1 s_2 A^{L^-}).
\]
	Then
\[
(a_1, \dots, a_{q+1}, b_1, \dots, b_p, -a_{2q+1}, \dots, -a_{q+2}, 0, a_{q+2}, \dots, a_{2q+1}, -b_p, \dots, -b_1, -a_{q+1}, \dots, -a_1)
\]
is Knuth equivalent to $\word(c \cdot A)$.
In particular this implies that $\word(A)$ is $\tau$-equivalent to $\word(c \cdot A)$.
\end{Theorem}
\begin{proof}
By Lemma \ref{L:LT} we have that
\[
(a_1, \dots, a_{q+1}, b_1, \dots, b_p, -a_{2q+1}, \dots, -a_{q+2}, 0, a_{q+2}, \dots, a_{2q+1}, -b_p, \dots, -b_1, -a_{q+1}, \dots, -a_1)
\]
is Knuth equivalent to
\[
(a_1, \dots, a_{q+1}, b_1, \dots, b_p, a_{q+2}, \dots, a_{2q+1}, 0, -a_{2q+1}, \dots, -a_{q+2}, -b_p, \dots, -b_1, -a_{q+1}, \dots, -a_1).
\]
Now if $b_{i_1}, \dots, b_{i_{q+1}}$ best fits under $a_1, \dots, a_{q+1}$ then we get that
this is Knuth equivalent to
\begin{align*}
&(a'_1, \dots, a'_p,
a_{q+2}, \dots, a_{2q+1},
b_{i_1}, \dots, b_{i_{q+1}},
0,  \\
&\qquad -b_{i_{q+1}}, \dots, -b_{i_1},
-a_{2q+1}, \dots, -a_{q+2},
-a'_p, \dots, -a'_1).
\end{align*}
Note that $a'_p = a_{q+1}$ or $a'_p = b_j < 0$ for some $j$, so in either case
we can best fit
\[
(b_{i_1}, \dots, b_{i_{q+1}},
0,
-b_{i_{q+1}}, \dots, -b_{i_3})
\]
under
\[
(a'_1, \dots, a'_p,
a_{q+2}, \dots, a_{2q+1})
\]
to get that
\[
(a'_1, \dots, a'_p,
a_{q+2}, \dots, a_{2q+1},
b_{i_1}, \dots, b_{i_{q+1}},
0,
-b_{i_{q+1}}, \dots, -b_{i_1},
-a_{2q+1}, \dots, -a_{q+2},
-a'_p, \dots, -a'_1)
\]
is Knuth equivalent to
\[
(a_1, \dots, a_{2q+1}, b_1, \dots, b_m,0,
-b_{i_{q+1}}, \dots, -b_{i_1},
-a_{2q+1}, \dots, -a_{q+2},
-a'_p, \dots, -a'_1).
\]
Now we can best fit $b_{m-q+2}, \dots, b_m, 0, -b_{i_{q+1}}, \dots, -b_{i_1}$ over
$-a_{2q+1}, \dots, -a_{q+2},
-a'_p, \dots, -a'_1)$
to get that
\[
(a_1, \dots, a_{2q+1}, b_1, \dots, b_m,0,
-b_{i_{q+1}}, \dots, -b_{i_1},
-a_{2q+1}, \dots, -a_{q+2},
-a'_p, \dots, -a'_1).
\]
is Knuth equivalent to
\[
(a_1, \dots, a_{2q+1}, b_1, \dots, b_m,0,
-b_m, \dots, -b_1,
-a_{2q+1}, \dots, -a_{1}).
\]
\end{proof}

Our goal is to prove that $L(A)$ is finite dimensional if and only if $A$ is $C$-conjugate to a row equivalent to column strict diagram.
The following lemmas build up to this.

\begin{Lemma} \label{L:alm}
Let $A \in sTab^{\le}(P)$.
 If $A^{L^-}$ is row equivalent to column strict then so is $A$.
\end{Lemma}
\begin{proof}
Recall that $A$ has row lengths given by $(2p+1, 2q+1, 2p+1)$.
By permuting entries within rows, we can find $p$ columns of $A^{L^-}$ which are strictly decreasing.
Furthermore, the entry in the bottom row of $A^{L^-}$ which is not in one of these columns must be negative.
By putting this entry below $0$ in $A$, and its negation above $0$, we can find a row equivalence class of
$A$ where every column left of $0$ contains one of the decreasing columns from $A^{L^-}$, and every column right
of zero is the reverse of the negation of one of the columns left of $0$.
Thus every column in this element of the row equivalence class of $A$ is strictly decreasing.
\end{proof}

\begin{Lemma}
   Let $A \in \sTab^{\leq}(P)$ and let $\bq = \part(\RS(A))$.
    If $\content(\bq) = \content(\bp)$ then
    $\bq = (2q+1,2p+1,2p+1)$ or $\bq = (2q+1,2p+2,2p)$.
\end{Lemma}
\begin{proof}
   Note that $\content(2q+1,2p+1,2p+1) = (p, p+1, q+1)$, and the only other partition
   with this content is $(2q,2p+2,2p+1)$.  Now
    by Lemma \ref{L:altRS}, $\part(\RS(A)) \geq (2q+1,2p+1,2p+1)$, thus
     $\part(\RS(A)) \neq (2q,2p+2,2p+1)$.
\end{proof}

By Theorem \ref{T:recs} we have that if $\part(\RS(A)) = (2q+1,2p+1,2p+1)$, then $A$ is
row equivalent to column strict. So we need only consider the case that $\part(\RS(A)) = (2q+1,2p+2,2p)$.

\begin{Lemma} \label{L:cB}
Let $A \in \sTab^{\leq}(P)$ and suppose that  $\part(\RS(A)) = (2q+1,2p+2,2p)$.  Then:
\begin{enumerate}
\item
  $A^{L^+}$ is row equivalent to column strict.
\item The middle row of $s_2 s_1 s_2 A^{L^+}$ contains only negative numbers.
\item The negation of the element in the bottom right position of $s_2 s_1 s_2 A^{L^+}$
is larger than the element in the upper right position of $s_2 s_1 s_2 A^{L^+}$.
Thus $c \cdot A$ is defined.
\item
$c \cdot A$ is row equivalent to column strict.
\end{enumerate}
\end{Lemma}
\begin{proof}
    Let $a_{-p}, \dots, a_{-1}, a_0, a_1, \dots, a_{p}$ be the increasing entries in the first row of $A$, and let
$-b_{q}, \dots, -b_1, 0, b_1, \dots, b_{q}$ be the middle row of $A$.

First we prove that $a_{-p}, \dots, a_0$ must best fit over $-b_{q}, \dots, -b_1$.
If it does not, then there must exists $i \in \{0, \dots, p\}$ such that $a_{-(p-i)} < -b_{q-i}$.  Thus we can form the following
increasing string in $\word(A)$:
\[
   a_{-p}, \dots, a_{-(p-i)}, -b_{q-i}, \dots, -b_1,0,b_1, \dots, b_{q-i}, -a_{-(p-i)}, \dots, -a_{-p}.
\]
This string has length $2 q + 3$, which contradicts
$\part(\RS(A)) = (2 q + 1,2 p+2,2p)$.

Next we prove that $a_1, \dots, a_p$ best fits over $b_1, \dots, b_q$.
If it did not, then there exists $i \in \{1, \dots, p\}$ such that $a_i < b_i$.  Thus we can form the following
increasing string in $\word(A)$:
\[
   a_{-p}, \dots, a_0, \dots, a_i, b_i, \dots, b_q.
\]
This string has length $p+q+2$, and
we can use it to find the following increasing string in $\word(A)$ of length $2p+2q+4$:
\[
   a_{-p}, \dots, a_0, \dots, a_i, b_i, \dots, b_q, -b_q, \dots, -b_i, -a_1, \dots, -a_0, \dots, -a_{-p}.
\]
This contradicts
$\part(\RS(A)) = (2 q + 1,2 p+2,2p)$.

Now we assume, for a contradiction, that $A^{L^+}$ is not row equivalent to column strict.
Let $j_0, \dots, j_p$ be positive integers such that
$-b_{j_p}, \dots, -b_{j_0}$ best fit under $a_{-p}, \dots, a_0$.
Let $i$ be the smallest non-negative integer such that $-b_{j_i} < -a_{i+1}$.  Such an $i$ must exists, since otherwise $A^{L^+}$ will be
row equivalent to column strict.
Define $b_0 = 0$.
Now let $k$ be the smallest integer such that
\begin{enumerate}
 \item $0 \leq k \leq i$;
\item $j_{i-l} = j_i - l$ if $0 < l \leq k$.
\item $j_{i-k-1} \neq j_{i-k}-1$.
\end{enumerate}
This implies that $a_{-(i-k)} < -b_{j_{i-k}-1}$.
So we can form the following two disjoint increasing substrings in $\word(A)$:
\[
  a_{-p}, \dots, a_{-(i-k)}, -b_{j_{i-k}-1}, \dots, b_{-1}, 0, b_1, \dots, b_q
\]
and
\[
  -b_q, \dots, -b_{j_i}, -a_{i+1}, \dots, -a_1, -a_0, -a_{-1}, \dots, -a_{-p}.
\]
The first string has length $p-i+k+1+j_{i-k}-1+1+q = p+q-i+k+j_{i-k}+1$.
The second string has length $q-j_i+1 + i+1 + p+1 = q+p-j_i+i+3$.
Thus, using the fact that $j_{i-k} = j_i-k$, the combined length of these two strings
is $2q + 2p + 4$, which contradicts $\part(\RS(A)) = (2q+1, 2p+2, 2 p)$.
Thus $A^{L^+}$ is row equivalent to column strict.

Finally we need to prove that the middle row of
$s_2 s_1 s_2 A^{L^+}$ contains only negative numbers.
Let $j_1, \dots, j_p$ be such that $-b_{j_p}, \dots, -b_{j_1}$ best fit over $-a_p, \dots, -a_1$.
Now it is clear that all the numbers in the last row of $s_2 A^{L^+}$ are negative.
Now let $a'$ be the entry in the first row of $A^{L^+}$ which does not best fit over
$-b_{j_p}, \dots, -b_{j_1}$.
If $a' >0$, then since all the $-b_i$'s are negative, we must have that $a' = a_0$.
In this case for $i=1, \dots, p$, $(a_{-i}, -b_{j_i}, -a_i)$ is a decreasing string in
$\word(A^{L^+})$ and in $\word(A)$.  Furthermore, reversing and negating these strings yields a further $p$ disjoint deceasing strings
of length 3
in $\word(A)$.  These, and the string $(a_0, 0, -a_0)$ show that
$\part(\RS(A))^T$ is larger than a partition of the form $(3^{2p+1}, *)$.
This contradicts $\part(\RS(A)) = (2q+1, 2p+2, 2p)$.
So we have that $a' < 0$, and furthermore the middle row of $s_1 s_2 A^{L^+}$ contains only negative numbers.
Now since the last row of $s_1 s_2 A^{L^+}$ also contains all negative numbers, we have that
the middle row of $s_2 s_1 s_2 A^{L^+}$ contains only negative numbers.

Now let $x$ be the element in the upper right position of
$s_2 s_1 s_2 A^{L^+}$, and let
$y$ be the element in the lower right position.
We need to show that $x < -y$.
If $a_0 < 0$ then this is clear since in this case every element of $A^{L^+}$ is negative.
When $a_0 >0$, we need to consider the bottom row of
$s_2 s_1 s_2 A^{L^+}$.  This row will contain $-a_n, \dots, -a_1$, and it will also contain
$-b_i$, where $-b_i$ is not one of the elements of $-b_m, \dots, -b_1$ which best fits over
$-a_n, \dots, -a_1$.
Let $-b_{k_n}, \dots, -b_{k_1}$ be as above, ie they are the elements which best fit over
$-a_n, \dots, -a_1$.
Note that $-b_{k_n}, \dots, -b_{k_1}$ are the elements in the middle row of $s_2 A^{L^+}$.
Now let $a'$ be the element in the first row of $A^{L^+}$ which is one of the elements which best fit over
$-b_{k_n}, \dots, -b_{k_1}$, so the middle row of
$s_1 s_s A^{L^+}$ contains $-a', -b_{k_n}, \dots, -b_{k_1}$.
We have already proved that since $a_0 > 0$, $a' < 0$.
So $a' < -b_{k_1}$, since otherwise $a_0$ would be the element which did not best fit over $-b_{k_n}, \dots, -b_{k_1}$.
So $-b_i < a' < -b_{k_1}$.  This implies that $-b_i < -a_1$, since otherwise
$-b_{k_1}$ would not be the element which best fits over $-a_1$.
Thus $-a_1$ is the element in the bottom right position of $s_2 s_1 s_2 A^{L^+}$, and
$a_0$ is the element in the upper right position of $s_2 s_1 s_2 A^{L^+}$, and we already have
that $a_0 < a_1$.

To see that $c \cdot A$ is row equivalent to column strict, simply note that
$(c \cdot A)^{L^-} = s_2 s_1 s_2 A^{L^+}$ is row equivalent to column strict, and apply Lemma \ref{L:alm}.
\end{proof}

Now we can state the main theorem of this section, which is analogous to Theorem \ref{T:caction1}.
The proof is very similar, where Lemma~\ref{L:cB} plays the role of Lemma \ref{L:cdefined}, and so omitted

\begin{Theorem} \label{T:caction2}
Suppose that $l$ is odd and $l > m$, and let
$A \in \sTab^\leq(P)$.
   Then $L(A)$ is finite dimensional if and only if
$A$ is $C$-conjugate to an s-table that is justified row equivalent to column strict.
   Furthermore, if $L(A)$ is finite dimensional, then
   $c \cdot L(A) \cong L(c \cdot A)$.
\end{Theorem}

Last in this section we give the following technical lemma, which is needed in the
proof of Theorem \ref{T:mainintro}.

\begin{Lemma} \label{L:Blargersmaller}
If $A \in \sTab^\leq(P)$ is row equivalent to column strict then
   $\word(c \cdot A)$ can be obtained from $\word(A)$ through a series of
Knuth equivalences and larger-smaller transpositions.
In particular, $\part(\RS(A)) \leq \part(\RS(c \cdot A))$.
\end{Lemma}
\begin{proof}
    Let
$(a_1, \dots, a_{2q+1}, b_1, \dots, b_p, 0, -b_p, \dots, -b_1, -a_{2q+1}, \dots, -a_1) = \word(A)$.
Due to Theorem \ref{T:LT3},
since $\word(A^{L^-}) = (a_1, \dots, a_q, b_1, \dots, b_p, -a_{2q+1}, \dots, -a_{q+1})$,
it suffices to show that
\[
(a_1, \dots, a_q, b_1, \dots, b_p, -a_{2q+1}, \dots, -a_{q+1},0,a_{q+1}, \dots, a_{2q+1}, -b_p, \dots, -b_1, -a_q, \dots, -a_1)
\]
can be obtained from
$\word(A)$  by a sequence of larger-smaller transpositions and Knuth equivalences.
First we can swap $a_{2 q+1}$ with its right neighbour, and $-a_{2q+1}$ with its left neighbour repeatedly until
we get a word with $a_{2 q+1}, 0 -a_{2 q+1}$ in the middle, then we can swap $a_{2 q+1}$ with $0$, then swap $a_{2 q+1}$ with $-a_{2 q+1}$,
then swap $0$ with $-a_{2 q+1}$ so that
we have $-a_{2 q+1}, 0, a_{2 q+1}$ in the middle of our word.
Now we can repeat this process with $a_{2q}$ and $a_{2q}$, then $a_{2q-1}$ and $-a_{2q-1}$, and so on.
Eventually, since $a_{q+1} > 0$, we will get
\[
(a_1, \dots, a_q, b_1, \dots, b_p, -a_{2q+1}, \dots, -a_{q+1},0,a_{q+1}, \dots, a_{2q+1}, -b_p, \dots, -b_1, -a_q, \dots, -a_1).
\]
\end{proof}

\section{The general case} \label{S:general}

Now we return to the case of general $\bp$ as in \eqref{e:bp}.  As usual $P$ is
the symmetric pyramid of $\bp$ with rows labelled $1, \dots, r, 0, -r, \dots, -1$ from
top to bottom.

\subsection{The component group action} \label{SS:compgroup}
In this section we describe the action of the component group $C$ on the subset
of $\Tab^\le(P)$ corresponding to finite dimensional $U(\g,e)$-modules.
The discussion here is completely analogous to the situation for even
multiplicity nilpotent elements as described in \cite[\S5.5]{BG2}, so we
are quite brief.  We use the notation for the component group $C$
from \S\ref{ss:comp}.

The operation of $c$ has been defined on three row s-tables
in \S\ref{S:Ccase} and \S\ref{S:Bcase}, and
this can be extended to any s-table by just
acting on the middle three rows. To define the action of the $c_k$
we proceed in exact analogy with \cite[\S5.5]{BG2}.  That is we use row swapping
operations $\bar s_i \star$ to move row $i_k$ to row $r$, then we apply $c$
and then we apply the reverse row swaps.  So for $A \in \Tab^{\le}(P)$, and $\tau
= \bar s_{i_k} \bar s_{i_k+1} \dots \bar s_{r-1} \in \bar S_r$ we have
$$
c_k \cdot A = \tau^{-1} \star ( c \cdot (\tau \star A)).
$$
Of course, this will not be defined for all $A \in \sTab^{\le}(P)$, but the
following proposition can be proved in the same way as \cite[Proposition 5.5]{BG2},
and we require Proposition \ref{P:changehw} for the proof.

\begin{Proposition} \label{P:caction3}
Let $A \in \sTab^\leq(P)$ and suppose that $L(A)$ is finite dimensional.
Then $c_k \cdot A$ is defined and $L(c_k \cdot A) \cong c_k \cdot L(A)$.
\end{Proposition}

\subsection{Proof of main theorem} \label{ss:generalproof}

Now we are in a position to prove Theorem \ref{T:mainintro}.

\begin{proof}[Proof of Theorem \ref{T:mainintro}]
The statement in the theorem about the component group action is given by Proposition \ref{P:caction3}.

Suppose that $A$ is justified row equivalent to column strict. Then $L(A)$ is finite dimensional
by Theorems \ref{T:recs}, \ref{T:BV} and \ref{T:BGKt2} and
thus $b \cdot L(A)$ is finite dimensional for any $b \in C$ by Proposition \ref{P:caction3}.

We are left to prove that if $L(A)$ is finite dimensional, then $A \in \sTab^\cc(P)$.
We prove this by induction on $r$.  The case $r=0$ is trivial,
and the case $r=1$ is given by Lemmas \ref{L:recsC} and \ref{L:recsB} and
Theorems \ref{T:caction1} and \ref{T:caction2}.

Now assume that $L(A)$ is finite dimensional and $r \ge 2$.
Using an inductive argument based on ``Levi
subalgebras'' of $U(\g,e)$ just as in the proof of \cite[Theorem 5.13]{BG1}
we may assume that $A^2_{-2}$ is justified row equivalent to
column strict, where $A^2_{-2}$ denotes the s-table
obtained from $A$ by removing rows 1 and $-1$.
Also by Lemma \ref{L:Aplus} we have that $A_r^1$ is justified row equivalent to column strict,
where $A_r^1$ is the table formed by rows 1 to $r$ of $A$.

Therefore, we can permute entries in the left justification of $A^2_{-2}$ so that all the columns
are strictly decreasing.  Furthermore, we can place each of the entries in row $1$ of $A$ over
a column
so each entry is larger than the entry immediately
below it.
Then we can place each of the entries of  row $-1$ of $A$ under a column in the left justification
of $A^2_{-2}$ so that each entry is smaller than the entry above it; and we can do this skew symmetrically
in the sense that if $a$ is an entry in row $1$ of $A$ and $a$ is placed over a column whose top entry is $b$,
then we can place $-a$ under a column whose bottom entry is $-b$.  Let $A_l$ denote the resulting diagram.

Let $\bq = \part(\RS(A))$.  As explained below the conditions above along with the
Theorem \ref{T:BV} give restrictions on the possibilities
for $\bq$.  The proof is completed with combinatorial arguments that show
that either $\bq = \bp$, or that $i_1 = 1$, and that $\part(\RS(c_1 \cdot A)) = \bp$.
So that by Theorem~\ref{T:recs}, either $A$ or $c_1 \cdot A$ is row equivalent to column strict.

In the diagram $A_l$, let $x$ be the number of columns which go through all the rows,
let $y$ columns which go through all the rows except the top row (so $y$ is also the number of
columns which go through all the rows except the bottom row),
and let $z$ be the number of columns which go through all the rows except the top and bottom row.
Further, let $u$ be the number of columns which go through all the rows except the middle row,
let $v$ be the number of columns which go through all the rows except
the top row and the middle row (so $v$ is also the number of
columns which go through all the rows except the middle row and the bottom row),
and let $w$ be the number of columns which go through all the rows except the top, middle, and bottom rows.
Note that $x+y+u+v= p_1$, and $x+2y+z+u+2v+w = p_2$.
So we have $x$ strictly decreasing columns of length $2r+1$,
$2y + u$ strictly decreasing columns of length $2r$,
$z + 2v$ strictly decreasing columns of length $2r-1$,
and $w$ strictly decreasing columns of length $2r-2$.

By counting the lengths of the other columns in $A_l$ similarly, and
using Lemma \ref{L:altcRS} we can conclude that
\[
\bq^T  \geq ((2r+1)^x,(2r)^{2y+u},(2r-1)^{z+2v},(2r-2)^w,(2r-4)^{p_3-p_2}, \dots, 2^{p_{r-1}-p_r})
\]
if $p_0 \leq p_{r-1}$;
and
\begin{align*}
\bq^T  \geq &((2r+1)^x,(2r)^{2y},(2r-1)^{z},(2r-3)^{p_3-p_2},
\dots, (2r-2k+5)^{p_{k-1}-p_{k-2}}, \\
  &\quad
(2r-2k+3)^{p_0-p_k-1},
(2r-2k+2)^{p_k-p_0}, (2r-2k)^{p_{k+1}-p_k},
\dots, 2^{p_{r-1}-p_r}).
\end{align*}
if $p_0 > p_{r-1}$ and $k$ is the number such that $p_i \ge p_0$ if and only if $i \ge k$
(note that in this case we have $u=v=w=0$).

Thus we get that
\[
   \bq \leq (p_r^2, p_{r-1}^2, \dots, p_2^2, p_2-w, x+2y+u, x) \quad \text{if $p_0 \leq p_{r-1}$}
\]
and
\[
   \bq \leq (p_r^2, p_{r-1}^2, \dots, p_{k-1}^2, p_0, p_k^2, \dots, p_2^2, p_1-z, x) \quad \text{if $p_0 > p_{r-1}$.}
\]

Since we also have $\bp$ as a lower bound of $\bq$, this implies that
\[
   \bq = (p_r^2, p_{r-1}^2, \dots, p_2^2, a, b, c) \quad \text{if $p_0 \leq p_{r-1}$}
\]
and
\[
   \bq = (p_r^2, p_{r-1}^2, \dots, p_{k}^2, p_0, p_{k-1}^2, \dots, p_2^2, a, b) \quad \text{if $p_0 > p_{r-1}$,}
\]
for positive integers $a,b,c$.
Since $\content(\bq) = \content(\bp)$,
we get a very limited number of possibilities for $a,b,c$,
as explained below.

From now we restrict to the case $\g = \sp_{2n}$ in this proof, as the
case $\g = \so_{2n+1}$ is entirely similar; in some places we would require references
from Section \ref{S:Bcase} rather than Section \ref{S:Ccase}.

We know that $p_0$ must be even.
If $p_0 < p_1$ and $p_1$ is even, then we must have
$(a,b,c) = (p_1,p_1,p_0)$
or $(a,b,c) = (p_1+1,p_1-1,p_0)$.
If $p_0 < p_1$ and $p_1$ is odd, then
$(a,b,c) = (p_1,p_1,p_0)$.
If $p_1 < p_0 < p_2$, then $p_1$ is even and $(a,b,c) = (p_0,p_1,p_1)$.
Finally, if $p_0 > p_2$, then $p_1$ is even and $(b,c) = (p_1,p_1)$.

By Theorem \ref{T:recs} if $\bq = \bp$, then
$A$ is justified row equivalent to column strict, and we are done.
So for the rest of this proof we will assume that $\bq \neq \bp$,
so
we are assuming that
$p_0 < p_r$, $p_r$ is even, and
\begin{equation} \label{EQ:q1}
   \bq = (p_r^2, \dots, p_2^2, p_1+1, p_1-1, p_0).
\end{equation}
It is be useful to record that
\begin{equation} \label{EQ:qT1}
\bq^T =  ((2r+1)^{p_0}, (2r)^{p_1-p_0-1}, (2r-1)^2, (2r-2)^{p_2-p_1-1}, (2r-4)^{p_3-p_2}, \dots, 2^{p_{r-1}-p_r}).
\end{equation}
Let $\sigma = s_{r-1} \dots \bar s_2 \bar s_1$ and $A' = \sigma \star A$,
and $\RS(A') = \RS(A)$ by Proposition \ref{P:changehw}.
Then the lengths of the middle three rows of
$A'$ are given by
$p_1,p_0,p_1$.
Let $B$ be the middle 3 rows of $A'$.

We claim that $\part(\RS(B)) = (p_1+1,p_1-1,p_0)$.
To see this,
first we suppose that
$\part(\RS(B)) = (p_r,p_r,p_0)$.
Then $B$ is justified row equivalent to column strict.
Now, since $(A')^1_r$ is justified row equivalent to column strict,
this allows us to find $p_0$ disjoint decreasing words
of length $2r+1$, which are disjoint from a further $p_1-p_0$ disjoint decreasing words
of length $2r$.
Thus by Lemma \ref{L:altcRS}
$\bq^T \geq ((2r+1)^{p_0}, (2r)^{p_1-p_0}, *)$, which contradicts \eqref{EQ:qT1}.
Now we also cannot have that any part of $\part(\RS(B))$ is larger than
$p_r+1$, since then we could use the fact that all the rows of $A'$ are increasing to conclude that
$\part(\bq)$ would be strictly larger than
a partition of the form
\[
   (p_1^2, p_2^2, \dots, p_{r-1}^2, p_r+1, *),
\]
which contradicts
\eqref{EQ:q1}.

Now we have, by Theorem \ref{T:caction1},
that $\part(\RS(c \cdot B)) = (p_0,p_1,p_1)$.
We also have by Lemma \ref{L:Clargersmaller}
that $\part(\RS(c_1 \cdot A)) \leq \part(\RS(A))$.

We need to argue that we can find enough maximal or near maximal length descending chains in $c \cdot A'$
to force $\RS(c_1 \dot A)$ to have shape $\bp$.
We have, by Lemma \ref{L:Aplus}, that $(c \cdot A')^1_r$ is justified row equivalent to column strict.
Further, by Theorem \ref{T:recs}, we have that $c \cdot B$ is justified row equivalent to column strict.

We can find $p_0$ descending strings of length 3 and $p_1-p_0$ strings of length 2, and all these strings
start in row $r$ and end in row $-r$.  Since $(c \cdot A')^1_r$ is justified row equivalent to column strict,
it has $p_1$ strings of length $r$ ending in row $r$, and $(c \cdot A')^{-r}_{-1}$ has $p_1$ strings of length $r$
starting in row $-1$.  So we can glue these strings together along their entries in rows $1$ and $-1$
to obtain $p_0$ disjoint decreasing strings of length $2r+1$ which are disjoint from $p_1-p_0$ disjoint decreasing strings of length $2r$.
So if $\bq' = \part(\RS(c \cdot A))$, we can conclude that
$\bq'^T \geq ((2r+1)^{p_0}, (2r)^{p_1-p_0},*)$, which implies that $\bq' = \bp$, so $c_1 \cdot A$ is justified row equivalent to column strict, as required.
\end{proof}

Finally, this theorem along with
Theorems \ref{T:BGKt1} and \ref{T:recs2} immediately imply the following classification of the primitive ideals
with associated variety equal to $\overline{G \cdot e}$.
\begin{Corollary}
The set of primitive ideals with associated variety $\overline{G \cdot e}$
   is equal to
\[
    \{ \Ann_{U(\g)} L(\lambda_A) \mid A \in \sTab^\cc(P) \}.
\]
\end{Corollary}

\appendix

\section{An alternative version of the Barbasch--Vogan algorithm}

In this appendix we consider the alternative version of the Barbasch--Vogan
algorithm for $\so_{2n+1}$ mentioned in \S\ref{ss:BV} above.  Our main result is Corollary \ref{C:BV}, which
shows that this adapted version gives the same output as the original version.  Below
we recall the algorithm then in the subsequent subsections construct
the  machinery required to prove Corollary \ref{C:BV}.

Some terminology and notation used in this section is as follows. By a {\em Young
diagram} we mean a finite collection of boxes, or cells, arranged in left-justified rows, with the row lengths weakly decreasing.  We often identify a Young diagram with its underlying partition.  A tableau is a filling of a
Young diagram by integers with weakly increasing rows and strictly decreasing columns.  We write $\part(T)$ for the
partition underlying a tableau $T$.  The Robinson--Schensted algorithm is denoted by $\RS$.

\subsection{The algorithms}

We need to define the {\em content of a partition}.
Let $\bq = (q_1 \le q_2 \le \dots, \le q_m)$ be a partition.
By inserting $0$ at the beginning if necessary, we may assume that $m$ is odd.
Let $(s_1, \dots, s_k)$, $(t_1, \dots, t_l)$ be such that
as unordered lists,
$(q_1, q_2+1, q_3+2, \dots, q_r+r-1)$ is equal to
$(2 s_1, \dots, 2 s_k, 2 t_1 +1, \dots, 2 t_l+1)$.
Now we define the content of $\bq$ to be the unordered list
\[
  \content(\bq) = (s_1, \dots, s_k, t_1, \dots, t_l).
\]

We now state the Barbasch--Vogan algorithm from \cite{BV1a} for the case
$\g = \so_{2n+1}$ in purely combinatorial terms.

\noindent {\bf Algorithm:}

\noindent {\em Input:} $\aa = (a_1 , \dots , a_n , -a_n , \dots , -a_1 )$ a skew symmetric string of
integers.

\noindent {\em Step 1:} Calculate $\bq =
\part(\RS(a_1 , \dots , a_n , -a_n , \dots , -a_1 ))$.

\noindent {\em Step 2:}
Calculate $\content(\bq)$. \\
Let $(u_1 \leq \dots \leq u_{2k+1})$ be the sorted list with the same entries as $\content(\bq)$. \\
For $i=1,\dots, k+1$ let $s_i = u_{2 i-1}$. \\
For $i=1,\dots, k$ let $t_i = u_{2i}$.

\noindent {\em Step 3:}
Form the list $(2 s_1+1, \dots, 2 s_{k+1}+1, 2 t_1,\dots, 2 t_{k})$. \\
In either case let $(v_1< \dots < v_k)$ be this list after sorting.

\noindent {\em Output:} $\BV(\aa) = \bq' = (v_1, v_2-1, \dots, v_{2k+1}-2k)$.

\medskip

The modified version is denoted by $\BV'$ and works in exactly the same way as $\BV$ except that
it calculates
$\RS(a_1 , \dots , a_n , 0, -a_n , \dots , -a_1 )$
instead of
$\RS(a_1 , \dots , a_n , -a_n , \dots , -a_1 )$ in Step 1.

\subsection{Domino Tableaux}
We require some facts about domino tableaux, which we collate below.

There are two types of domino tableaux, those with an even number of boxes and those with an odd number of boxes.
A {\em domino tableau with an even number of boxes} is a Young diagram that has been tiled with $2 \times 1$ and $1 \times 2$ dominos,
where each domino is labeled with a positive integer, such that the rows are increasing and
the columns are decreasing.
A {\em domino tableau with an odd number of boxes} is the same as a domino tableau with an even number of boxes,
except it also has a $1 \times 1$ box labeled with $0$, which must necessarily occur in the lower left position.

For example,
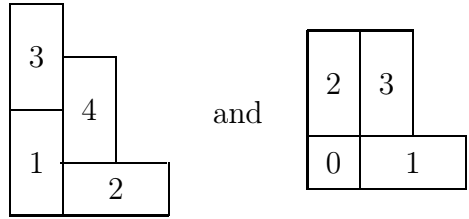
\[
\begin{array}{c}
\begin{picture}(60,80)
   \put(0,0){\line(1,0){60}}
   \put(19,20){\line(1,0){40}}
   \put(0,40){\line(1,0){20}}
   \put(20,60){\line(1,0){20}}
   \put(0,80){\line(1,0){20}}
   \put(0,0){\line(0,1){80}}
   \put(20,0){\line(0,1){80}}
   \put(20,20){\line(0,1){40}}
   \put(40,20){\line(0,1){40}}
   \put(60,0){\line(0,1){20}}
 \put(10,20){\makebox(0,0){{1}}}
 \put(40,10){\makebox(0,0){{2}}}
 \put(30,40){\makebox(0,0){{4}}}
 \put(10,60){\makebox(0,0){{3}}}
\end{picture}
\end{array}
\quad \text{and} \quad
\begin{array}{c}
\begin{picture}(60,60)
   \put(0,0){\line(1,0){60}}
   \put(0,20){\line(1,0){60}}
   \put(0,60){\line(1,0){40}}
   \put(0,0){\line(0,1){60}}
   \put(20,0){\line(0,1){60}}
   \put(40,20){\line(0,1){40}}
   \put(60,0){\line(0,1){20}}
 \put(10,10){\makebox(0,0){{0}}}
 \put(40,10){\makebox(0,0){{1}}}
 \put(10,40){\makebox(0,0){{2}}}
 \put(30,40){\makebox(0,0){{3}}}
\end{picture}
\end{array}
\]
are domino tableaux.

Given a domino tableau $R$, we let $\part(R)$ denote the partition underlying $R$, i.e.\ the partition
given by the row lengths of $R$.
We say that a partition has {\em domino shape} if it is the underlying partition of a domino tableau.

The following lemma is straightforward to prove by induction.

\begin{Lemma} \label{L:evenodd}
Let $\bp = (p_1 \leq p_2 \leq \dots \leq p_m)$ be a partition,
where $p_1$ may be $0$ and $m$ is odd.
Choose $r_1, \dots, r_k$ and $s_1, \dots, s_l$ so that
$(p_1, p_2+1, \dots, p_m+m-1)$ is equal to $(2 r_1, \dots, 2 r_k, 2 s_1 +1, \dots, 2 s_l +1)$ as unordered lists.
     If $\bp$ has domino shape and has an  even number of boxes, then $k=l+1$.
     If $\bp$ has domino shape and has an odd number of boxes, then $k+1=l$.
\end{Lemma}

Let $T$ be a tableau whose boxes are labelled by the integers $-n, \dots, -1, 1, \dots, n$ or
the integers $-n,\dots, -1, 0, 1, \dots, n$.
We recall an algorithm $\DT$ which takes as input such a tableau and outputs a domino tableau;
it was defined in \cite{BV1a}.
To define $\DT(T)$, first note that
$-n$ must occur in the lower left corner of $T$.  Swap $-n$ with the smaller of its neighbours which lie above or
to the right of $-n$.  Continue swapping $-n$ with its smaller neighbour which is either above or right of it.
If the last number that $-n$ is swapped with is not $n$ then we say that $\DT(T)$ is undefined.
Otherwise replace the squares with $-n$ and $n$ with a domino containing
$n$.
Now repeat this procedure for $1-n, 2-n, \dots, -1$ treating any
squares which have been replaced with dominos as if they were not present.
If for any $i$ the last number that $-i$ is replaced with is not $i$ then $\DT(T)$ is undefined.
Otherwise we eventually get a domino tableau.

For example, suppose
\[
 T =
\RS(-2,-3,1,0,-1,3,2)  =
\begin{array}{c}
\begin{picture}(60,60)
   \put(0,0){\line(1,0){60}}
   \put(0,20){\line(1,0){60}}
   \put(0,40){\line(1,0){60}}
   \put(0,60){\line(1,0){20}}
   \put(0,0){\line(0,1){60}}
   \put(20,0){\line(0,1){60}}
   \put(40,0){\line(0,1){40}}
   \put(60,0){\line(0,1){40}}
 \put(8,10){\makebox(0,0){{-3}}}
 \put(28,10){\makebox(0,0){{-1}}}
 \put(50,10){\makebox(0,0){{2}}}
 \put(8,30){\makebox(0,0){{-2}}}
 \put(30,30){\makebox(0,0){{0}}}
 \put(50,30){\makebox(0,0){{3}}}
 \put(10,50){\makebox(0,0){{1}}}
\end{picture}
\end{array}.
\]
Now when we apply the above algorithm we first swap $-3$ with $-2$, then with $0$, then with $3$.
Now replace the boxes containing $3$ and $-3$ with a domino containing $3$.
This results in the following diagram:
\[
\begin{array}{c}
\begin{picture}(60,60)
   \put(0,0){\line(1,0){60}}
   \put(0,20){\line(1,0){60}}
   \put(0,40){\line(1,0){60}}
   \put(0,60){\line(1,0){20}}
   \put(0,0){\line(0,1){60}}
   \put(20,0){\line(0,1){60}}
   \put(40,0){\line(0,1){20}}
   \put(60,0){\line(0,1){40}}
 \put(8,10){\makebox(0,0){{-2}}}
 \put(28,10){\makebox(0,0){{-1}}}
 \put(50,10){\makebox(0,0){{2}}}
 \put(10,30){\makebox(0,0){{0}}}
 \put(40,30){\makebox(0,0){{3}}}
 \put(10,50){\makebox(0,0){{1}}}
\end{picture}
\end{array}.
\]
Now $-2$ first swaps with $-1$, then $2$, which results in the following diagram.
\[
\begin{array}{c}
\begin{picture}(60,60)
   \put(0,0){\line(1,0){60}}
   \put(0,20){\line(1,0){60}}
   \put(0,40){\line(1,0){60}}
   \put(0,60){\line(1,0){20}}
   \put(0,0){\line(0,1){60}}
   \put(20,0){\line(0,1){60}}
   \put(60,0){\line(0,1){40}}
 \put(8,10){\makebox(0,0){{-1}}}
 \put(40,10){\makebox(0,0){{2}}}
 \put(10,30){\makebox(0,0){{0}}}
 \put(40,30){\makebox(0,0){{3}}}
 \put(10,50){\makebox(0,0){{1}}}
\end{picture}
\end{array}.
\]
Finally $-1$ swaps with $0$ and then $1$, and the resulting domino tableau is
\begin{equation} \label{EQ:DT1}
  \DT(T) =
\begin{array}{c}
\begin{picture}(60,60)
   \put(0,0){\line(1,0){60}}
   \put(0,20){\line(1,0){60}}
   \put(20,40){\line(1,0){40}}
   \put(0,60){\line(1,0){20}}
   \put(0,0){\line(0,1){60}}
   \put(20,0){\line(0,1){60}}
   \put(60,0){\line(0,1){40}}
 \put(10,10){\makebox(0,0){{0}}}
 \put(40,10){\makebox(0,0){{2}}}
 \put(40,30){\makebox(0,0){{3}}}
 \put(10,40){\makebox(0,0){{1}}}
\end{picture}
\end{array}.
\end{equation}

Let $W$ is the Weyl group of type $B_n$ acting on $\{\pm 1,\dots, \pm n\}$ is the
natural way.  Then the image of $(-n, \dots, -1, 1, \dots, n)$ under the action of
some $\sigma \in W$ is called a {\em signed permutation of
$(-n, \dots, -1, 1, \dots, n)$}.  A {\em signed permutation of
$(-n, \dots, -1, 0,1, \dots, n)$} is defined similarly.

The following lemma follows from
\cite[Proposition 2.3.3]{Lea} and
\cite[Theorem 4.1.1]{Lea}.

\begin{Lemma}
    If $(a_1, \dots, a_n, -a_n, \dots, -a_1)$ is a signed permutation of
$(-n, \dots, -1, 1, \dots, n)$,
and
     $(b_1, \dots, b_n, 0,-b_n, \dots, -b_1)$ is a signed permutation of
$(-n, \dots, -1, 0,1, \dots, n)$,
then both $\DT(\RS(a_1, \dots, a_n, -a_n, \dots, -a_1))$ and $\DT(\RS(b_1, \dots, b_n, 0,-b_n, \dots, -b_1))$ are defined.
\end{Lemma}

We may identify $W$ with the signed permutations of
$(-n, \dots, -1, 1, \dots, n)$ or the signed permutations of
$(-n, \dots, -1, 0,1, \dots, n)$.
Under this identification, we consider the algorithms of Garfinkle defined in \cite[\S2]{Ga1a}
to map a signed permutation of
$(-n, \dots, -1, 1, \dots, n)$
or
$(-n, \dots, -1, 0,1, \dots, n)$
to a domino tableau.
We denote these versions of Garfinkle's algorithm by $\operatorname{G}_0$ and $\operatorname{G}_1$ respectively.

The following proposition is
\cite[Proposition 4.2.3]{Lea}:
\begin{Proposition} \label{P:Le} $ $
\begin{enumerate}
\item[(i)]If $w$ is a signed permutation of
$(-n, \dots, -1, 1, \dots, n)$,
then
$\DT(\RS(w)) = \operatorname{G}_0(w)$.
\item[(ii)] If $w$ is a signed permutation of
$(-n, \dots, -1, 0,1, \dots, n)$,
then
$\DT(\RS(w)) = \operatorname{G}_1(w)$.
\end{enumerate}
\end{Proposition}

Our aim is to show that
$\part(\RS(a_1, \dots, a_n,-a_n, \dots, -a_1))$
has the same content as
$\part(\RS(a_1, \dots, a_n,0,-a_n, \dots, -a_1))$.
We do this by exploiting the results in \cite{Pi}, which explain how to relate
$\operatorname{G}_0(a_1, \dots, a_n,-a_n, \dots, -a_1)$
and
$\operatorname{G}_1(a_1, \dots, a_n,0,-a_n, \dots, -a_1)$.
To explain these results,
we need to define the {\em cycles} of a standard domino tableau.
This requires a few other definitions as well.

We define coordinates on a Young diagram by labeling its rows and columns.
We declare that the bottom row is row $1$, the row above the bottom is row $2$, and so on.
We declare that the left most column is column $1$, the column to its right is column $2$, and so on.
Now we say the box in position $(i,j)$ is {\em fixed} if $i+j$ is odd and the diagram has
an even number of boxes or if $i+j$ is even and the diagram has an odd number of boxes.

Let $R$ be a domino tableau,
and let $D(k)$ be a domino with label $k$ in $R$.
If the fixed coordinate of $D(k)$ occurs in the lower box or right box of $D(k)$,
let $E$ denote the square below and to the right of the fixed coordinate of $D(k)$.
If the fixed coordinate of $D(k)$ occurs in the upper box or left box of $D(k)$,
let $E$ denote the square above and to the left of the fixed coordinate of $D(k)$.
We label $E$ with the integer $m$ determined via
\[
  m =
   \begin{cases}
         l & \text{if $E$ is a square in $R$ and $l$ is the label of $E$'s square in $R$;} \\
         -1 & \text{if either coordinate of $E$ is $0$}; \\
          \infty & \text{if $E$ lies above or to the right of $R$}.
   \end{cases}
\]
Now we define $D'(k)$ to be a domino which contains two squares, one in the fixed position of $D(k)$,
and the other which is adjacent to $E$ and so that the subdiagram containing
$D'(k)$ and $E$ has decreasing columns and increasing rows.

For example, if
\[
R =
  \begin{array}{c}
     \begin{picture}(80,40)
      \put(0,0){\line(1,0){80}}
      \put(40,20){\line(1,0){40}}
      \put(0,40){\line(1,0){40}}
      \put(0,0){\line(0,1){40}}
      \put(20,0){\line(0,1){40}}
      \put(40,0){\line(0,1){40}}
      \put(80,0){\line(0,1){20}}
         \put(10,20){\makebox(0,0){{1}}}
         \put(30,20){\makebox(0,0){{2}}}
         \put(60,10){\makebox(0,0){{3}}}
     \end{picture}
  \end{array},
\]
then
$D'(1)$ is a domino which occupies positions $(2,1)$ and $(3,1)$,
$D'(2)$ is a domino which occupies positions $(1,2)$ and $(1,3)$,
$D'(3)$ is a domino which occupies positions $(1,4)$ and $(1,5)$.

Suppose a domino tableau is labeled with $\{1, \dots, n\}$.
We use this to generate an equivalence relation on $\{1, \dots, n \}$
via
$i \sim j$ if $D(j)$ and $D'(i)$ share a box.
The {\em cycles of a domino tableau} are the equivalence classes of this equivalence relation.
For example, if $R$ is as above then
the cycles of $R$ are $\{1\}$ and $\{2,3\}$.

If $R$ is a domino tableau with an even number of boxes and $c$ is a cycle of $R$,
then we can define a new domino tableau $R' = \operatorname{MT}(R,c)$
by replacing $D(k)$ with $D'(k)$ for every $k \in c$.
This will remove one box and add one box to the underlying Young diagram of $R$.
If the box removed is in position $(1,1)$ then we put a box with $0$ in position $(1,1)$ of $R'$,
so that we do in fact get another domino tableau.
For example if $R$ is as above then
\[
  \operatorname{MT}(R,\{1\}) =
  \begin{array}{c}
     \begin{picture}(80,60)
      \put(0,0){\line(1,0){80}}
      \put(0,20){\line(1,0){20}}
      \put(40,20){\line(1,0){40}}
      \put(20,40){\line(1,0){20}}
      \put(0,60){\line(1,0){20}}
      \put(0,0){\line(0,1){60}}
      \put(20,0){\line(0,1){60}}
      \put(40,0){\line(0,1){40}}
      \put(80,0){\line(0,1){20}}
         \put(10,10){\makebox(0,0){{0}}}
         \put(10,40){\makebox(0,0){{1}}}
         \put(30,20){\makebox(0,0){{2}}}
         \put(60,10){\makebox(0,0){{3}}}
     \end{picture}
  \end{array},
\]
and
\[
  \operatorname{MT}(R,\{2,3\}) =
  \begin{array}{c}
     \begin{picture}(100,40)
      \put(0,0){\line(1,0){100}}
      \put(20,20){\line(1,0){80}}
      \put(0,40){\line(1,0){20}}
      \put(0,0){\line(0,1){40}}
      \put(20,0){\line(0,1){40}}
      \put(60,0){\line(0,1){20}}
      \put(100,0){\line(0,1){20}}
         \put(10,20){\makebox(0,0){{1}}}
         \put(40,10){\makebox(0,0){{2}}}
         \put(80,10){\makebox(0,0){{3}}}
     \end{picture}
  \end{array}.
\]

Observe that the operator $\operatorname{MT}$ removes a box and adds a box to the Young diagram underlying $R$,
and that the removed box is either in position $(1,1)$, or is a {\em removable box} of $R$, that is
if it is removed you still have a valid Young diagram.

A key feature of $\operatorname{MT}$ is that is does not change the content
of the underlying partition:
\begin{Theorem} \label{T:content}
   Let $R$ be a domino tableau with an even number of boxes, let $c$ be a cycle of $R$, and let
   $\bp = \part(R)$ and $\bq = \part(\operatorname{MT}(R,c))$.
   Then $\content(\bp) = \content(\bq)$.
\end{Theorem}

\begin{proof}
   First we rule out the case  that
    $\bp$ has an odd number (say $2m+1$) of parts and $\bq$ has one more part than $\bp$.
   Suppose, for a contradiction, that $\bq$ has $2m+2$ parts, so the top row of $\operatorname{MT}(R,c)$ has one box.
   Let $D'(k)$ be the domino in $\operatorname{MT}(S,c)$ which covers this box.  So the box in the fixed position of $D'(k)$ must be the box
   in position $(2m+1,1)$, which is a contradiction since $2m+1+1$ is even.

   Next we rule out the case that $\bp$ has an even number (say $2m$) of parts
  and $\bq$ has one less part than $\bp$.
   Suppose this is the case, so the top row $\bp$ has length one,
   so there must be a domino $D(k)$ which occupies positions $(2m-1,1)$ and $(2m,1)$ of $\bp$.
   Now $(2m,1)$ is the fixed position of this domino, so $D'(k)$ will also have a box in
    position $(2m,1)$, hence $\bq$ has at least $2m$ parts, which is a contradiction.
   Thus we have established that the number of integers in $\content(\bq)$ is the same as the number of integers in $\content(\bp)$.

   Let $\bp = (p_1 \leq \dots \leq p_{2m+1})$, where $p_1$ may be $0$.
   Now we consider the case that $\operatorname{MT}(R,c)$ has the same number of boxes as $R$.
   Let $\bq = (q_1 \leq \dots \leq q_{2m+1})$, where $q_1$ may be $0$.
   So we must have that $q_i = p_i$ except for $i=j$ and $i=k$ for some
   integers $j,k$ where $j \neq k$, and $q_j = p_j+1$ and $q_k = p_k-1$.
    By Lemma \ref{L:evenodd} we have that one of $p_j+j-1, p_k+k-1$ must be even and one must be odd, because otherwise
    $(q_1, q_2+1, \dots, q_{2m+1}+2m)$ would not have one more even element than odd elements.
     The box at position $(j,p_j)$ of $\operatorname{MT}(R,c)$ is the box which gets added to the Young diagram of $R$.
     Thus this box is a box which is in $D'(k)$ but is not in $D(k)$.
      This implies that this box is not the box in fixed position in $D'(k)$, thus $p_j+j$ is even, so $p_j+j-1$ must be odd,
       and $p_k+k-1$ is even.
      This implies that $\bp$ and $\bq$ have the same content.

   Now we consider the case that $\operatorname{MT}(R,c)$ has one more box than $R$.
   Let $\bq = (q_1 \leq \dots \leq q_{2m+1})$, where $q_1$ may be $0$.
   So we must have that $q_i = p_i$ except for $i=j$ for some integer $j$, where $q_j = p_j+1$.
   Note that $p_j+j-1$ must be even since $(q_1, q_2+1, \dots, q_{2m+1}+2m)$ must have one more odd number than even number.
      This implies that $\bp$ and $\bq$ have the same content.
\end{proof}

For a list of cycles $c_1, \dots, c_m$ of a domino tableau $R$ with an even number of boxes let
$R_i = \operatorname{MT}(R_{i-1},c_i)$, where $R_0 = R$.
Now let $\operatorname{MT}(R,c_1, \dots, c_m) = R_m$.

The following theorem is a less specific version of \cite[Theorem 3.1]{Pi}:
\begin{Theorem} \label{T:Pi}
Let
$$R = \operatorname{G}_0(a_1, \dots, a_n, -a_n, \dots, -a_1),$$ and
$$R' = \operatorname{G}_1(a_1, \dots, a_n, 0,-a_n, \dots, -a_1),$$
where $(a_1, \dots, a_n, -a_n, \dots, -a_1)$ be a signed permutation of $(-n, \dots, -1, 1, \dots, n)$.
Then there exists cycles $c_1, \dots, c_m$ of $R$ such that
$R' = \operatorname{MT}(R,c_1, \dots, c_m)$.
\end{Theorem}

Now we get the following corollary.

\begin{Corollary} \label{C:BV}
Let $\aa = (a_1, \dots, a_n, -a_n, \dots, -a_1)$ be a skew symmetric string of
integers.  Then
$\operatorname{BV}(\aa) = \operatorname{BV}'(\aa)$.
\end{Corollary}

\begin{proof}
This follows from Proposition \ref{P:Le}, and Theorems \ref{T:content} and \ref{T:Pi} for the case
where $\aa$ is signed permutation of $(-n, \dots, -1, 1, \dots, n)$.  The general case follows, because $\bq = \RS(a_1 , \dots , a_n , -a_n , \dots , -a_1 )$
and $\bq' = \RS(a_1 , \dots , a_n , 0, -a_n , \dots , -a_1 )$ depend only on the relative order of the $a_i$, so that we may replace $(a_1 , \dots , a_n , -a_n , \dots , -a_1 )$
 by signed a permutations without altering $\bq$ or $\bq'$.
\end{proof}

\end{document}